\documentclass[12pt,reqno]{amsart}
\advance\hoffset by -0.5 truecm
\usepackage{amssymb}
\newtheorem{Theorem}{Theorem}[section]
\newtheorem{Lemma}[Theorem]{Lemma}
\newtheorem{Sublemma}[Theorem]{Sublemma}
\newtheorem{Corollary}[Theorem]{Corollary}

\newtheorem{Proposition}[Theorem]{Proposition}
\newtheorem{Conjecture}[Theorem]{Conjecture}

\newtheorem{Remark}[Theorem]{Remark}

\def \dim{{\mbox {dim}}\,}

\begin{document}
\title[On the eigenvalues of the Laplacian on fibred manifolds]{On the eigenvalues of the Laplacian on fibred manifolds}

\author[C. Sung]{Chanyoung Sung}
 \address{Dept. of Mathematics Education  \\
          Korea National University of Education \\
          Cheongju, Korea
          }
\email{cysung@kias.re.kr}

\keywords{Laplacian eigenvalue, fiberwise symmetrization, Riemannian submersion}

\subjclass[2020]{58E99, 35P15, 49R05}

\date{}

%\maketitle

\begin{abstract}
We prove various comparison theorems of the $i$-th eigenvalue $\lambda_i$ of the Laplacian on fibred Riemannian manifolds by using fiberwise spherical and Euclidean (or hyperbolic) symmetrization. In particular we generalize the Lichnerowicz inequality and the Faber-Krahn inequality to fiber bundles, and prove a counterpart to Cheng's $\lambda_1$ comparison theorem under a lower Ricci curvature bound.

By applying these, it is shown that $\lambda_1,\cdots,\lambda_k$ of a fiber bundle given by a Riemannian submersion with totally geodesic fibers of sufficiently positive Ricci curvature are respectively equal to $\lambda_1,\cdots,\lambda_k$ of its base, and $\lambda_1$ of a (possibly singular) fibration with Euclidean subsets as fibers is no less than $\lambda_1$ of the disk bundle obtained by replacing each fiber with a Euclidean disk of the same dimension and volume.
%provide many examples for it including $n$-sphere bundles and projective bundles such as geometrically ruled surfaces and blow-ups of $\Bbb CP^n$ along a linear subspace where $n\geq 2$.
\end{abstract}

\maketitle

\tableofcontents

\section{Introduction and statement of main results}

Eigenvalues and eigenfunctions of the Laplace-Beltrami operator (or Laplacian) on a Riemannian manifold have been of crucial importance not only in geometry but also in physics and engineering.(\cite{Berard, Courant-Hilbert, Shanker}) They carry important information of a Riemannian manifold which can not be captured by curvature.

The Laplacian $\Delta_{(M,g)}$ of a smooth closed Riemannian $m$-manifold $(M^m ,g)$ is a linear differential operator acting on $C^\infty(M)$ defined by $$\Delta u:=-\textrm{div}(\nabla u)=-\sum_{i=1}^m(V_i\circ V_i-\nabla_{V_i}V_i)(u)$$ for $u\in C^\infty(M)$ where $V_1,\cdots,V_m$ is a local orthonormal frame. A maximal set of its eigenfunctions forms a complete orthonormal basis of a Hilbert space $L^2(M)$, and the $i$-th eigenvalue $\lambda_i(M,g)$ of $\Delta_{(M,g)}$ can be represented as the minimum of the Rayleigh quotient on the $L^2$-orthogonal complement to the subspace generated by eigenfunctions corresponding to $\lambda_0,\cdots,\lambda_{i-1}$, where the Rayleigh quotient $R_g(u)$ of $u\in L_1^2(M)$ is given by $$R_g(u)=\frac{\int_M |du|_g^2\ d\mu_{g}}{\int_M u^2\ d\mu_{g}}$$ and its minimizer is the $i$-th eigenfunction.
We alert the reader that our indexing of eigenvalues allows $\lambda_i$ to be equal to $\lambda_{i+1}$ which is repeated the number of times equal to its multiplicity, and $\lambda_0$ is always 0. In particular $\lambda_1(M,g)$ is represented as
\begin{eqnarray*}
\lambda_1&=&\min\{R_g(u)|u\in L_1^2(M)\backslash\{0\}, \int_Mu\cdot 1\ d\mu_{g}=0\}.
\end{eqnarray*}
%where $L_1^2(M)$ is the completion of $C^\infty(M)$ w.r.t a Sobolev $L_1^2$-norm.

Due to this variational characterization, various upper bounds for $\lambda_1$ can be obtained, and more importantly there are also many results about the lower bounds of $\lambda_1$ in terms of topology and geometric quantities such as curvature, injectivity radius, volume, diameter, etc.(\cite{SY}), in which case the use of isoperimetric inequalities are substantially effective. The most well-known theorem of this kind is the Lichnerowicz-Obata theorem \cite{Lich, obata-1962} which states that if $m\geq 2$ and Ricci curvature $Ric_g \geq m-1$, then $$\lambda_1(M^m,g)\geq m=\lambda_1(S^m,g_{_{\Bbb S}})$$ where the equality holds iff $(M,g)$ is isometric to $\Bbb S^m:=(S^m,g_{_{\Bbb S}})$.
Throughout the paper $g_{_{\Bbb S}}$ denotes the standard round metric of constant curvature 1 and $V_m$ denotes the volume of $\Bbb S^m$. There are many exact computations of $\lambda_1$ in homogeneous manifolds \cite{BLP, BP, Tani, Tanno-0, Tanno, urakawa-0, urakawa} in which cases the spectrum can be obtained (in principle) algebraically from its corresponding Lie group, isotropy subgroup, and inner product on its Lie algebra, but in general it's very hard to compute $\lambda_1$ even on well-known manifolds with canonical Riemannian metrics.

We are concerned with the product formula of $\lambda_i$. The case of a Riemannian product is easily obtained by the Hilbert space isomorphism $$L^2(N\times M)\simeq L^2(N)\otimes L^2(M)$$ for any smooth closed Riemannian manifold $(N,h)$. Namely $\lambda_j(N,h)+\lambda_k(M,g)$ constitute eigenvalues of $(N\times M,h+g)$, so in particular
\begin{eqnarray}\label{pizzaplanet-0}
\lambda_1(N\times M,h+g)=\min(\lambda_1(N,h),\lambda_1(M,g)).
\end{eqnarray}
Then one naturally asks about the next canonical one, a warped product or more generally a Riemannian submersion.(\cite{Bergery-Bour, Besson, Ejiri, Gilkey-Park, Tsukada}) It is well-known that for a Riemannian submersion $\pi:X\rightarrow B$ defined on a smooth closed manifold $X$ with minimal fibers, the pullback of any eigenfunction of $B$ by $\pi^*$ is also eigenfunction of $X$ with the same eigenvalue. (Refer to Watson \cite{Watson} and also Berger, Gauduchon, and Mazet \cite{BGM}.)
Thus we have an upper estimate of $\lambda_1$ (in fact all $\lambda_i$) on a Riemannian submersion with totally geodesic fibers. To find a lower estimate, we shall use the method of fiberwise spherical symmetrization developed in \cite{Sung} and obtain a generalized Lichnerowicz inequality on fiber bundles.
\begin{Theorem}\label{Main-theorem}
Let $(F,g_{_F})$ be a smooth closed Riemannian $m$-manifold for $m\geq 2$ with $Ric_{g_{_F}} \geq m-1$, $(B,g_{_B})$ be a smooth closed Riemannian $n$-manifold, and $\check{\rho}$ be a smooth positive function on $B$.
Suppose that $\pi: (X,g_{_X})\rightarrow (B,g_{_B})$ is a smooth Riemannian submersion with totally geodesic fibers isometric to $(F,g_{_F})$, and
$g_{_X}^\rho$ for $\rho=\check{\rho}\circ \pi$ is a warped variation of $g_{_X}$.
Then $$\lambda_1(X,g_{_X}^\rho) \geq  \lambda_1(B\times S^m, g_{_B}+\check{\rho}^2g_{_{\Bbb S}}).$$
% \left(\frac{V}{V_m}\right)^{\frac{2}{m}}\lambda_1(B\times S^m,(\frac{V}{V_m})^{\frac{2}{m}}g_{_B}+\textsl{g}_V),$$ where $V$ is the volume of $(F,g_{_F})$.
\end{Theorem}
The above warped variation of a Riemanian submersion refers to the metric obtained by conformally rescaling each fiber by the given function $\rho$, and its detailed account is given in Subsection \ref{Korean-nuclear}. Its typical example is a warped product metric $g_{_B}+\rho^2g_{_F}$ when the original metric $g_{_X}$ is a Riemannian product $g_{_B}+g_{_F}$. In case of a warped product we already obtained the same result in \cite{Sung} by using fiberwise spherical symmetrization which is not directly applicable to the general case of nontrivial bundles.

In another special case when $\rho$ is constant, i.e. $g_{_X}^\rho$ is a canonical variation of $g_{_X}$, we can also compute higher eigenvalues $\lambda_i$ of $X$.
\begin{Theorem}\label{Main-Corollary}
Let $(F,g_{_F}), (B,g_{_B}), (X,g_{_X})$ and $\rho$ be as in Theorem \ref{Main-theorem}. Further assume that $n\geq 1$, $\rho=\textrm{constant}$, and for an integer $I\geq 1$ $$\lambda_I(B,g_{_B})\leq \frac{m}{\rho^2}.$$ Then for $i\leq I$ $$\lambda_i(X,g_{_X}^\rho)=\lambda_i(B, g_{_B})$$ and the pull back of the $i$-th eigenfunction of $(B, g_{_B})$ by $\pi^*$ is the $i$-th eigenfunction of $(X,g_{_X}^\rho)$.
\end{Theorem}
The above theorem enables us to exactly compute $\lambda_i$ on many nontrivial fiber bundles with totally geodesic fibers of sufficiently positive Ricci curvature. As shown in Section \ref{Main-Corollary-Example}, typical examples are geometrically ruled surfaces, nontrivial $S^m$ bundles $S^n\tilde{\times} S^m$ over $S^n$ for $m\geq 2$, generalized Berger spheres or squashed spheres, squashed $\Bbb CP^{2q+1}$, the unit sphere bundle of a real vector bundle of rank $\geq 3$, the projectivization of a complex vector bundle of rank $\geq 2$, etc, where metrics are given by Riemannian submersions with totally geodesic fibers sufficiently squashed. Many interesting manifolds belong to the above list. To name a few, some exotic 7-spheres are $S^3$-bundles over $S^4$ and the blow-up of $\Bbb CP^n$ along any linear subspace is realized as a projective bundle over a projective space. To the best of the author's knowledge, this is the first exact computations of $\lambda_i$ on these examples which are not homogeneous manifolds.

For a compact Riemannian manifold with boundary, eigenvalues of the Laplacian can be studied in a similar way to a closed manifold when Dirichlet or Neumann boundary conditions are imposed. We wish to obtain similar comparison theorems of the first Dirichlet eigenvalue, i.e. the least positive eigenvalue $\lambda_1$ for eigenfunctions $u$ satisfying Dirichlet boundary condition $$u|_{\partial X}=0$$ on a compact fibred manifold $X$ with boundary.

The prototypical comparison theorem for this case is the Faber-Krahn inequality \cite{Faber,Krahn} stating that for any bounded domain $X\subset \Bbb R^m$ and an open ball $X^\star\subset \Bbb R^m$ with the same volume as $X$
$$\lambda_1(X)\geq \lambda_1(X^\star).$$
To generalize this to a manifold fibred by bounded domains in $\Bbb R^m$, we perform fiberwise Euclidean symmetrization replacing each fiber with an Euclidean open ball of the same dimension and volume. The necessary analytic technicalities for this symmetrization method were also established in \cite{Sung}, and we follow the notations there. Throughout the paper the open ball(or disc) in $(\Bbb R^m,g_{_{\Bbb E}})$ with volume $v$ and center at the origin is denoted by $D^m_v$.
\begin{Theorem}\label{selfish}
Let $(N^n,h)$ be a smooth closed Riemannian $n$-manifold and $X$ be a bounded domain with piecewise-smooth boundary in the Riemannian product $N\times \Bbb R^m$.
Suppose that the $m$-dimensional volume $V(s)$ of $X\cap (\{s\}\times \Bbb R^m)$ for each $s\in N$ is a piecewise-smooth continuous positive function of $s\in N$.
Then for $$X^\star:=\{(s,q)\in N\times \Bbb R^m|q\in D^m_{V(s)}\}$$
the first Dirichlet eigenvalue satisfies $$\lambda_1(X)\geq \lambda_1(X^\star).$$
\end{Theorem}
Here $X^\star$ is called the fiberwise Euclidean symmetrization of $X$ and by the piecewise-smoothness of the function $V$ it is meant that $X^\star$ has the corresponding piecewise-smooth boundary so that its first Dirichlet eigenvalue is represented as the minimum of the Rayleigh quotient on $L_{0,1}^2(X^\star)\backslash\{0\}$ where $L_{0,1}^2(X^\star)$ is the completion of $C_c^\infty(X^\star)$ with respect to an $L_1^2$-norm.

As will be shown in Subsection \ref{selfish-man}, Theorem \ref{selfish} can be applied to many interesting manifolds $X$ with a singular fibration onto $N$, where fibers need not be all diffeomorphic but change topology.  For example $X$ can be a 3-manifold like a doughnut made of Swiss cheese and the corresponding $X^\star$ is a solid torus with varying cross-sectional area.

The analogous statement also holds for the hyperbolic symmetrization of $X\subset (N\times \Bbb H^m,h+g_{_{\Bbb H}})$
% and $$X^\bullet:=\{(s,q)\in N\times \Bbb H^m|q\in \mathfrak{D}^m_{V(s)}\}$$
where $\Bbb H^m$ is the $m$-dimensional open unit ball with center at the origin and Poincar\'{e} metric $g_{_{\Bbb H}}$.
% and $\mathfrak{D}^m_v$ denote the open $m$-disk in $(\Bbb H^m,g_{_{\Bbb H}})$ with center at the origin and hyperbolic volume $v$.
Similarly for $X\subset N\times M$ where $M$ has positive Ricci curvature, we have :
\begin{Theorem}\label{Kane}
Let $(N^n,h)$ and $(M^m,g)$ be smooth closed Riemannian manifolds where $M$ has volume $V$ and $Ric_g\geq m-1$ for $m\geq 2$.
Suppose that $X$ is a domain with piecewise-smooth boundary in the Riemannian product $N\times M$ such that the $m$-dimensional volume $V(s)\in (0,V)$ of $X\cap (\{s\}\times M)$ for each $s\in N$ is a piecewise-smooth continuous function of $s\in N$.
Then for $$X^*:=\{(s,q)\in N\times \Bbb S^m|r_{_{\Bbb S}}(q)< R(s)\}$$ where $r_{_{\Bbb S}}(q)$ is the distance of $q$ and the south pole in $\Bbb S^m$ and $R(s)$ is defined by the condition that the volume of a geodesic ball of radius $R(s)$ in $\Bbb S^m$ is $\frac{V(s)}{V}V_m$,
the first Dirichlet eigenvalue satisfies $$\lambda_1(X)\geq \lambda_1(X^*).$$
\end{Theorem}
An immediate consequence of this theorem is a kind of counterpart to Cheng's eigenvalue comparison theorem \cite{Cheng} stating that
$$\lambda_1(B_r^M)\leq \lambda_1(B_r^{\Bbb S})$$ where $B_r^M$ and $B_r^{\Bbb S}$ respectively denote a geodesic ball of radius $r$ in $(M^m,g)$ and $\Bbb S^m$ given as above. Precisely we have the following corollary where $\textrm{Vol}(\cdot)$ denotes volume.
\begin{Corollary}
Let $(N^n,h)$ and $(M^m,g)$ be as in Theorem \ref{Kane}. If $r<\pi$ and $r'$ satisfy $$\frac{\textrm{Vol}(B_{r'}^M)}{V}=\frac{\textrm{Vol}(B_{r}^{\Bbb S})}{V_m},$$ then
the first Dirichlet eigenvalues satisfy $$\lambda_1(N\times B_{r'}^M)\geq \lambda_1(N\times B_r^{\Bbb S}).$$
\end{Corollary}

It seems that the fiberwise symmetrization method is applicable to the estimation of eigenvalues of the Laplacian on some other fiber bundles too. We shall discuss some open problems in the last section. Our method can be also applied to the estimation of lower bounds of Yamabe constants of fiber bundles. The case of a warped product metric on a product manifold with fiber of positive Ricci curvature was done in \cite{Sung-Yamabe}.

\section{Preliminaries}

\subsection{Conventions and notations}

Here are collected some common conventions and notations to be used throughout this paper.

First of all, every manifold is assumed to be smooth and connected unless otherwise stated.
When $X$ is a $k$-dimensional Riemannian manifold (possibly with boundary), $\mu(X)$ denotes the $k$-dimensional volume of $X$. The characteristic function of a subset $S$ is denoted by $\chi_{_S}$.

In case that a Lie group $G$ has a smooth left (or right) action on a smooth manifold $X$, it induces a left (or right) action on $C^\infty(X)$ respectively defined by
\begin{eqnarray}\label{BHCP}
(g\cdot F)(x):=F(g^{-1}\cdot x)\ \ \ \ \textrm{or}\ \ \ \ (F\cdot g)(x):=F(x\cdot g^{-1})
\end{eqnarray}
for $F\in C^\infty(X)$, $g\in G$ and $x\in X$. We also let $|G|$ denote the cardinality of $G$ for finite $G$ and if $G$ is infinite, it denotes the volume $\mu(G)$ of $G$ with respect to a Riemannian metric under consideration.

Given a map $f$, its inverse image of a set $S$ is denoted by $f^{-1}S$, and
the open ball of radius $\epsilon>0$ and center $p$ in any metric space is denoted by $B_\epsilon(p)$.
%The closure of a subset $S$ in a topological space is usually denoted by $\overline{S}$, and also by $\textrm{cl}(S)$ when $S$ is too long.

\subsection{Weak derivative and Sobolev space}
On a smooth Riemannian manifold $(M,g)$, $L_1^p(M)$ for $p\geq 1$ is the completion of $$\{f\in C^\infty(M)| \ ||f||_{L_1^p}:=(\int_M (|f|^p+|df|_g^p)\ d\mu_g)^{\frac{1}{p}}<\infty\}$$ with respect to $L_1^p$-norm $||\cdot||_{L_1^p}$.
If $M$ is compact, then $L_1^p(M)$ is not only independent of choice of metric $g$ but also characterized as the space of weakly differentiable functions on $M$ with bounded $L_1^p$-norm. A locally integrable function $f:M\rightarrow \Bbb R$, i.e. in $L^1_{loc}(M)$, is called weakly differentiable if it has all of its weak partial derivatives of 1st order in a local coordinate. Here a weak derivative is meant in distributional sense, and required to be in $L^1_{loc}$. If a weak derivative exists, then it is unique up to a set of measure zero. The following characterization for weak differentiability is well-known.
\begin{Proposition}\label{Sewoong}
Let $\Omega\subset \Bbb R^n$ be a domain. Then $f\in L^1_{loc}(\Omega)$ is weakly differentiable iff it is equivalent to a function that is absolutely continuous on almost every straight line
%\footnote{Here the absolutely continuity on a straight line $L$ is meant by the absolute continuity on any segment of $L$ contained in $\Omega$.}
parallel to the coordinate axes and whose 1st order partial derivatives (which consequently exist a.e. and are measurable) are locally integrable in $\Omega$.
\end{Proposition}
As a result, the weak 1st order partial derivatives of a weakly differentiable function is equivalent to its ordinary partial derivatives defined almost everywhere.
The following is a well-known elementary fact.
%, and  we may drop the subscript for a metric $g$ in the notation $|df|_g=|\nabla f|_g$ if no confusion arises.
\begin{Lemma}\label{Zenith}
On a smooth Riemannian $m$-manifold $(M,g)$, let $f:M\rightarrow \Bbb R$ be weakly differentiable. Then $f^+(x):=\max (f(x),0)$ and $|f|$ are weakly differentiable such that
$$df^+=\left\{
  \begin{array}{ll}
    df & \hbox{on $\{f > 0\}$} \\
    0 & \hbox{on $\{f \leq 0\}$}
  \end{array}
\right.$$
$$d|f|=\left\{
  \begin{array}{ll}
    df & \hbox{on $\{f > 0\}$} \\
    0 & \hbox{on $\{f = 0\}$} \\
    -df & \hbox{on $\{f < 0\}$}
  \end{array}
\right.$$
where equalities are meant in almost-everywhere sense.

If $\varphi, \psi\in L_1^2(M)$, then
$$\int_M (\varphi^2-\psi^2)\   d\mu_g\leq ||\varphi-\psi||_{L_1^2}(||\varphi||_{L_1^2}+||\psi||_{L_1^2})$$
and
$$\int_M (|d\varphi|^2-|d\psi|^2)\   d\mu_g\leq ||\varphi-\psi||_{L_1^2}(||\varphi||_{L_1^2}+||\psi||_{L_1^2}).$$
\end{Lemma}
\begin{proof}
To prove the 1st statement, it's enough to check it locally, in which case the proof is given in Lemma 7.6 of \cite{GT}.

The 2nd statement is very elementary, as seen in
\begin{eqnarray*}
\int_M (\varphi^2-\psi^2)\   d\mu_g &=& (||\varphi||_{L^2}-||\psi||_{L^2})(||\varphi||_{L^2}+||\psi||_{L^2})\\
&\leq& ||\varphi-\psi||_{L^2}(||\varphi||_{L_1^2}+||\psi||_{L_1^2})\\
&\leq& ||\varphi-\psi||_{L_1^2}(||\varphi||_{L_1^2}+||\psi||_{L_1^2})
\end{eqnarray*}
and
\begin{eqnarray*}
\int_M (|d\varphi|^2-|d\psi|^2)\   d\mu_g &=& (||\ |d\varphi|\ ||_{L^2}-||\ |d\psi|\ ||_{L^2})(||\ |d\varphi|\ ||_{L^2}+||\ |d\psi|\ ||_{L^2})\\
&\leq& (||\ |d\varphi|-|d\psi|\ ||_{L^2})(||\varphi||_{L_1^2}+||\psi||_{L_1^2})\\
&\leq& (||\ |d\varphi-d\psi|\ ||_{L^2})(||\varphi||_{L_1^2}+||\psi||_{L_1^2})\\
&\leq& ||\varphi-\psi||_{L_1^2}(||\varphi||_{L_1^2}+||\psi||_{L_1^2})
\end{eqnarray*}
where the 2nd inequality is due to the triangle inequality in $\Bbb R^m$.
\end{proof}

\begin{Proposition}\label{gamsa1}
Let $H$ be a compact Lie group with a smooth Riemannian metric and $(M,g)$ be a smooth closed Riemannian $m$-manifold. Suppose that $H$ has a smooth isometric left action on $M$. For $f\in L_1^2(M)$ define $$f_\circ(x):=\frac{\int_{H} (h\cdot f)(x)\ d\mu(h)}{|H|}$$ where $d\mu(h)$ denotes the volume element at $h\in H$.
Then $f_\circ\in L_1^2(M)$ satisfying $$||f_\circ||_{L_1^2}\leq  ||f||_{L_1^2}.$$  The same is true of a right action.
\end{Proposition}
\begin{proof}
We shall prove only for a left action. The case of a right action is proved in the same way.

Let $\Psi:C^\infty(M)\rightarrow C^\infty(M)$ be the map given by $\Psi(f)=f_\circ$.
We claim that $$||\Psi(f)||_{L_1^2}\leq ||f||_{L_1^2}$$ for $f\in C^\infty(M)$.

By applying H\"older inequality, Fubini's theorem, and the fact that the $H$-action is isometric,
\begin{eqnarray*}
\int_Mf_\circ^2\ d\mu_{g}&=& \int_M\left(\frac{\int_{H} h\cdot f\ d\mu(h)}{|H|}\right)^2 d\mu_{g}\\ &\leq& \int_M\frac{\int_{H} (h\cdot f)^2\ d\mu(h)}{|H|} d\mu_{g}\\ &=& \frac{1}{|H|}\int_H\int_M (h\cdot f)^2\ d\mu_{g}d\mu(h)\\ &=& \frac{1}{|H|}\int_H\int_M f^2\ d\mu_{g}d\mu(h)\\ &=&\int_Mf^2\ d\mu_{g},
\end{eqnarray*}
and
\begin{eqnarray*}
\int_M|df_\circ|^2\ d\mu_{g} &\leq& \int_M\frac{\int_{H} |d(h\cdot f)|^2\ d\mu(h)}{|H|} d\mu_{g}\\ &=& \frac{1}{|H|}\int_H\int_M |d(h\cdot f)|^2\ d\mu_{g}d\mu(h)\\ &=& \frac{1}{|H|}\int_H\int_M |df|^2\ d\mu_{g}d\mu(h)\\ &=&\int_M|df|^2\ d\mu_{g}
\end{eqnarray*}
where the inequality in the 1st line holds due to :
\begin{Lemma}\label{gamsa2}
$$|df_\circ|^2\leq \frac{\int_{H} |d(h\cdot f)|^2\ d\mu(h)}{|H|}.$$
\end{Lemma}
\begin{proof}
It is enough to show it at any point $p\in M$.
Let $(x_1,\cdots,x_m)$ be a normal coordinate around $p$. Then
\begin{eqnarray*}
|df_\circ(p)|^2&=&\sum_i (\frac{\partial f_\circ}{\partial x_i}(p))^2 \\ &=& \sum_i \left(\frac{1}{|H|}\int_{H}\frac{\partial}{\partial x_i}(h\cdot f)(p)\ d\mu(h)\right)^2\\ &\leq& \sum_i \frac{1}{|H|}\int_{H}\left(\frac{\partial (h\cdot f)}{\partial x_i}(p)\right)^2 d\mu(h)\\ &=&
 \frac{1}{|H|}\int_{H}|d(h\cdot f)(p)|^2\ d\mu(h)
\end{eqnarray*}
where the 2nd equality is the application of the Leibniz rule by using that $H$ is compact and $f$ is smooth.
\end{proof}

So the linear map $\Psi$ is uniformly continuous w.r.t. $L_1^2$-topology. Since $C^\infty(M)$ is a dense subset of $L_1^2(M)$, $\Psi$ extends to a unique endomorphism of $L_1^2(M)$ still satisfying $||\Psi||\leq 1$. This completes the proof.
\end{proof}
The following extension of the above proposition will be used later in proving Theorem \ref{Main-theorem}.
\begin{Proposition}\label{gamsa}
Let $M$ and $H$ be as above, and further assume that there exists a smooth $\ell$-form $\Omega$ on $M$ for $0\leq \ell\leq m-1$ such that $\Omega$ is left(or right)-invariant and at each point it can be represented as $e_1^*\wedge\cdots\wedge e_\ell^*$ for orthonormal cotangent vectors $e_i^*$. Then for any $f\in L_1^2(M)$ $$\int_M|df_\circ\wedge \Omega|^2\ d\mu_{g} \leq \int_M|df\wedge \Omega|^2\ d\mu_{g}.$$
\end{Proposition}
\begin{proof}
Again we only prove for a left action. When $\ell=0$, it reduces to the above case, so we consider for $\ell\geq 1$. First we prove for smooth $f$.

To prove the analogue of Lemma \ref{gamsa2}, let $p$ be any point of $M$ and $\Omega(p)$ be equal to $e_1^*\wedge\cdots\wedge e_\ell^*$ for mutually orthonormal $e_i^*\in T_p^*M$. Take a normal coordinate $(x_1,\cdots,x_m)$  around $p$ such that $e_i^*$ is dual to $\frac{\partial}{\partial x_i}|_p$  for $i=1,\cdots,\ell$.
Again by the Leibniz integral rule we can switch the differentiation and integration so that
\begin{eqnarray*}
|df_\circ(p)\wedge \Omega(p)|^2&=&\sum_{i=m-\ell+1}^{m} (\frac{\partial f_\circ}{\partial x_i}(p))^2 \\ &=& \sum_{i=m-\ell+1}^{m} \left(\frac{1}{|H|}\int_{H}\frac{\partial}{\partial x_i}(h\cdot f)(p)\ d\mu(h)\right)^2\\ &\leq& \sum_{i=m-\ell+1}^{m} \frac{1}{|H|}\int_{H}\left(\frac{\partial (h\cdot f)}{\partial x_i}(p)\right)^2 d\mu(h)\\ &=&
 \frac{1}{|H|}\int_{H}|d(h\cdot f)(p)\wedge \Omega(p)|^2\ d\mu(h),
\end{eqnarray*}
and hence
\begin{eqnarray*}
\int_M|df_\circ\wedge \Omega|^2\ d\mu_{g} &\leq& \int_M\frac{\int_{H} |d(h\cdot f)\wedge \Omega|^2\ d\mu(h)}{|H|} d\mu_{g}\\ &=& \frac{1}{|H|}\int_H\int_M |d(L_{h^{-1}}^*(f))\wedge L_{h^{-1}}^*\Omega|^2\ d\mu_{g}d\mu(h)\\  &=& \frac{1}{|H|}\int_H\int_M |L_{h^{-1}}^*(df\wedge \Omega)|^2\ d\mu_{g}d\mu(h)\\ &=& \frac{1}{|H|}\int_H\int_M |df\wedge \Omega|^2\ d\mu_{g}d\mu(h)\\ &=&\int_M|df\wedge \Omega|^2\ d\mu_{g}
\end{eqnarray*}
where $L_{h^{-1}}:M\rightarrow M$ denotes the diffeomorphism given by $x\mapsto h^{-1}\cdot x$.

Now we prove for general $f\in L_1^2(M)$.
\begin{Lemma}\label{jeonghoon}
Suppose $\varphi_n\in L_1^2(M)$ for $n\in \Bbb N$ be a sequence converging to $\varphi$ in $L_1^2(M)$. Then $||d\varphi_n\wedge \Omega||_{L^2}\rightarrow ||d\varphi\wedge \Omega||_{L^2}$ as $n \rightarrow \infty$.
\end{Lemma}
\begin{proof}
Note that for any 1-form $\alpha$ on $M$, $|\alpha|\geq |\alpha\wedge \Omega|$ at each point.
So $$||(d\varphi_n-d\varphi)\wedge \Omega||_{L^2}\leq ||d\varphi_n-d\varphi||_{L^2}\rightarrow 0$$ from which the desired statement follows.
\end{proof}
Take $f_n\in C^\infty(M)$ for $n\in \Bbb N$ such that $||f_n-f||_{L_1^2}<\frac{1}{n}$. By the above lemma, $||df_n\wedge \Omega||_{L^2}\rightarrow ||df\wedge \Omega||_{L^2}$ as $n \rightarrow \infty$.
By Proposition \ref{gamsa1}, $$||(f_n)_\circ-f_\circ||_{L_1^2}\leq ||f_n-f||_{L_1^2}<\frac{1}{n}$$ so that $$||d(f_n)_\circ\wedge \Omega||_{L^2}\rightarrow ||df_\circ\wedge \Omega||_{L^2}$$ by the above lemma.
Now using that the desired inequality hold for each $f_n$, we finally have
\begin{eqnarray*}
\int_M|df_\circ\wedge \Omega|^2\ d\mu_{g} &=& \lim_{n\rightarrow \infty} \int_M|d(f_n)_\circ\wedge \Omega|^2\ d\mu_{g}\\ &\leq& \lim_{n\rightarrow\infty} \int_M|df_n\wedge \Omega|^2\ d\mu_{g}\\ &=& \int_M|df\wedge \Omega|^2\ d\mu_{g}.
\end{eqnarray*}

\end{proof}

\subsection{Rayleigh quotient}
In this subsection we provide some standard facts about the Rayleigh quotient.
First we shall use the following elementary lemma very often when estimating the Rayleigh quotient.
\begin{Lemma}\label{Sunukjian}
Let $a\ne 0$ and $b$ be real constants. Then there exists a constant $C:=4|\frac{b}{a}|+3$ such that
$$|\frac{b}{a}-\frac{b+y}{a+x}|\leq C\max (|\frac{x}{a}|, |\frac{y}{a}|)$$ for any real $x,y$ such that $|\frac{x}{a}|<\frac{1}{2}$.
\end{Lemma}
\begin{proof}
By applying the mean value theorem to $f(t)=\frac{1}{1+t}$ for $|t|<\frac{1}{2}$, there exists $t_0$ such that $|t_0|\in (0,\frac{1}{2})$ and
\begin{eqnarray*}
\frac{1}{1+t}&=&f(0)+f'(t_0)(t-0)\\ &=& 1-\frac{1}{(1+t_0)^2}t.
\end{eqnarray*}
Here $t_0$ and hence $f'(t_0)$ depend on $t$, and $f'(t_0)$ ranges in $(f'(-\frac{1}{2}),f'(\frac{1}{2}))= (-4,-\frac{4}{9})$.

Using this, we have
\begin{eqnarray*}
\frac{b+y}{a+x}&=&\frac{\frac{b}{a}+\frac{y}{a}}{1+\frac{x}{a}}\\ &=& (\frac{b}{a}+\frac{y}{a})(1+\mu(x)\frac{x}{a})
\end{eqnarray*}
where $\mu(x)\in (-4,-\frac{4}{9})$, and therefore
\begin{eqnarray*}
|\frac{b+y}{a+x}-\frac{b}{a}|&\leq&  |\frac{b}{a}\mu(x)\frac{x}{a}|+|\frac{y}{a}(1+\mu(x)\frac{x}{a})|\\&\leq& 4|\frac{b}{a}\frac{x}{a}|+(1+4\frac{1}{2})|\frac{y}{a}| \\
&\leq& C\max (|\frac{x}{a}|, |\frac{y}{a}|).
\end{eqnarray*}

\end{proof}

\begin{Proposition}\label{yam-0}
Let $(M,g)$ be a smooth closed Riemannian $m$-manifold for $m\geq 1$ and
define $$\mathfrak{A}:=\{\varphi\in L_1^2(M)|\ ||\varphi||_{L_1^2(M)}<C_1, ||\varphi||_{L^2(M)}>C_2\}$$ where $C_i$ are positive constants.
Then there exist positive constants $\varepsilon$ and $\check{C}$ such that $||f||_{L_1^2(M)}<\varepsilon$ implies
$$|R_g(\varphi)-R_g(\varphi+f)|\leq \check{C}||f||_{L_1^2(M)}$$ for any $\varphi\in \mathfrak{A}$.
\end{Proposition}
\begin{proof}
First, for any $\varphi\in \mathfrak{A}$ $$|R_g(\varphi)|=\frac{\int_M |d\varphi|^2d\mu_g}{\int_M \varphi^2 d\mu_g}<\left(\frac{C_1}{C_2}\right)^2.$$
To estimate the difference between $R_g(\varphi)$ and
\begin{eqnarray*}
R_g(\varphi+f) &=& \frac{\int_M\langle d\varphi+df, d\varphi+df\rangle d\mu_g}{\int_M (\varphi+f)^2d\mu_g}\\ &=& \frac{\int_M |d\varphi|^2d\mu_g+2\langle d\varphi,df\rangle_{L^2} +\int_M |df|^2 d\mu_g}{||\varphi||^2_{L^2}+2\langle \varphi, f\rangle_{L^2} +||f||^2_{L^2}},
\end{eqnarray*}
we wish to apply the above lemma.

If $||f||_{L_1^2}< \min(1, (C_2)^2(4C_1+2)^{-1})$, then by H\"older inequality
$$|2\langle \varphi, f\rangle_{L^2} +||f||^2_{L^2}|\leq (2||\varphi||_{L^2}+||f||_{L^2})||f||_{L^2}\leq (2C_1+1)||f||_{L_1^2},$$
$$|2\langle d\varphi,df\rangle_{L^2} +\int_M |df|^2d\mu_g|\leq (2||\varphi||_{L_1^2}+||f||_{L_1^2})||f||_{L_1^2}\leq (2C_1+1)||f||_{L_1^2},$$ and hence
$$\frac{|2\langle \varphi, f\rangle_{L^2} +||f||^2_{L^2}|}{||\varphi||^2_{L^2}}\leq \frac{(2C_1+1)||f||_{L_1^2}}{(C_2)^2}<\frac{1}{2}.$$
So the application of Lemma \ref{Sunukjian} with $a=||\varphi||^2_{L^2}$ and $b=\int_M |d\varphi|^2d\mu_g$ gives
\begin{eqnarray*}
|R_g(\varphi)-R_g(\varphi+f)|&\leq& (4|\frac{b}{a}|+3)\max (\frac{(2C_1+1)||f||_{L_1^2}}{a}, \frac{(2C_1+1)||f||_{L_1^2}}{a})\\
&\leq& (4\left(\frac{C_1}{C_2}\right)^2+3)\frac{(2C_1+1)||f||_{L_1^2}}{(C_2)^2}
\end{eqnarray*}
which yields the desired inequality.
\end{proof}

\begin{Proposition}\label{yam-1}
On a smooth closed Riemannian $m$-manifold $(M,g)$ for $m\geq 1$, let $$\{\varphi_i\in L_1^2(M)|\int_M\varphi_i\ d\mu_g=0, \int_M\varphi_i^2\ d\mu_g=1, i\in \Bbb N\}$$ be a minimizing sequence for $R_g$, i.e. $\lim_{i\rightarrow \infty}R_g(\varphi_i)= \lambda_1(M,g)$.
%If $\{\varphi_i\}$ is uniformly bounded, i.e. $\sup_i||\varphi_i||_\infty< C$ for a constant $C$,
Then there exists the first eigenfunction $\varphi$ to which a subsequence of $\{\varphi_i\}$ converges in $L_1^2$-norm.
\end{Proposition}
\begin{proof}
From
\begin{eqnarray*}
\lim_{i\rightarrow \infty}\int_M|d\varphi_i|^2\ d\mu_g&=& \lambda_1(M,g),
\end{eqnarray*}
$\{\varphi_i|i\in \Bbb N\}$ is a bounded subset of $L_1^2(M)$. Then there exists $\varphi\in L_1^2(M)$ and a subsequence converging to $\varphi$ weakly in $L_1^2$ and also strongly in $L^2$ by the compact embedding $L_1^2\hookrightarrow L^2$.
A further subsequence converges to $\varphi$ pointwisely almost everywhere. By abuse of notation we just let $\{\varphi_i\}$ be the subsequence.

Since $\varphi_i\rightarrow \varphi$ in $L^2$,
$$
\int_M\varphi\ d\mu_g=0
\ \ \ \ \ \
\textrm{and}
\ \ \ \ \ \
\lim_{i\rightarrow \infty}\int_M \varphi_i\varphi\  d\mu_g=\int_M\varphi^2\ d\mu_g=\lim_{i\rightarrow \infty}\int_M \varphi_i^2  d\mu_g.
$$
By the weak convergence in $L_1^2$,
$$\int_M|d\varphi|^2d\mu_g+\int_M\varphi^2\ d\mu_g = \lim_{i\rightarrow \infty}(\int_M\langle d\varphi_i,d\varphi\rangle d\mu_g+\int_M \varphi_i\varphi\  d\mu_g),$$ so
\begin{eqnarray}\label{jongguk-3}
\int_M|d\varphi|^2d\mu_g &=& \lim_{i\rightarrow \infty}\int_M\langle d\varphi_i,d\varphi\rangle d\mu_g\\ &\leq& \limsup_{i\rightarrow \infty}(\int_M|d\varphi_i|^2d\mu_g    )^{\frac{1}{2}}(\int_M|d\varphi|^2d\mu_g)^{\frac{1}{2}}\nonumber\\ &=& \lim_{i\rightarrow \infty}(\int_M|d\varphi_i|^2d\mu_g    )^{\frac{1}{2}}(\int_M|d\varphi|^2d\mu_g)^{\frac{1}{2}}\nonumber
\end{eqnarray}
implying that
$$\int_M|d\varphi|^2d\mu_g \leq \lim_{i\rightarrow \infty}\int_M|d\varphi_i|^2d\mu_g.$$

Therefore $$\lambda_1(M,g)=\lim_{i\rightarrow \infty}\frac{\int_M|d\varphi_i|^2\ d\mu_g}{\int_M \varphi_i^2\ d\mu_g} \geq
\frac{\int_M|d\varphi|^2\ d\mu_g}{\int_M \varphi^2\ d\mu_g},$$
and hence this inequality must be an equality so that $\varphi$ is a 1st eigenfunction and $$\lim_{i\rightarrow \infty}\int_M|d\varphi_i|^2d\mu_g=\int_M|d\varphi|^2d\mu_g.$$

Combined with (\ref{jongguk-3}), it gives
\begin{eqnarray*}
\lim_{i\rightarrow \infty}\int_M|d\varphi-d\varphi_i|^2d\mu_g&=&\lim_{i\rightarrow \infty}\int_M(|d\varphi|^2-2\langle d\varphi_i,d\varphi\rangle+|d\varphi_i|^2)\ d\mu_g=0,
\end{eqnarray*}
completing the proof that $\varphi_i\rightarrow\varphi$ strongly in $L_1^2$.

\end{proof}

\subsection{Fiberwise symmetrization}

Here we summarize the fundamentals of fiberwise symmetrization which will be the essential tool for obtaining all our main results. Throughout the paper we denote the round $m$-sphere with volume $V$ by $S^m_V$ whose metric is thus given by $\textsl{g}_V:={\left( \frac{V}{V_m} \right) }^{\frac{2}{m}}g_{_{\Bbb S}}$.

A smooth function $F$ defined on a product manifold $N\times M$ where we always regard $M$ as fiber and $N$ as base is called a generic-fiberwise Morse function if there exists an open dense subset $N_0$ of $N$ with measure-zero complement such that $F$ restricted to each fiber over $N_0$ is a Morse function on $M$. One may take the largest possible $N_0$, i.e. the union of all such $N_0$, and we denote it by $N_0(F)$. It turns out that the subset of such functions are dense in $C^\infty(N\times M, \Bbb R)$ with $C^2$-topology.
\begin{Theorem}[\cite{Sung}]\label{morse-1}
Let $M$ and $N$ be smooth closed manifolds and $C^\infty(N\times M)$ be endowed with $C^2$-topology. If $F\in C^\infty(N\times M)$, then for any open neighborhood $U\subseteq C^\infty(N\times M)$ of $F$ there exist $\tilde{F}\in U$ which is a generic-fiberwise Morse function.

Under further assumption that a compact Lie group $G$ acts on $N$ and $M$ smoothly and $F:N\times M\rightarrow \Bbb R$ is invariant under this action, $\tilde{F}$ can be chosen such that the $G$-orbit of $\tilde{F}$ also remains in $U$. If the action on $M$ is trivial, then $\tilde{F}$ can be chosen to be $G$-invariant.
\end{Theorem}
Given a smooth generic-fiberwise Morse function $F :N\times M \rightarrow \Bbb R$, we define $$F_{\bar{*}}:N_0\times S^m_V\rightarrow \Bbb R$$ as the fiberwise spherical rearrangement of $F$; namely $F_{\bar{*}}|_{\{s\}\times S^m_V}$ for $s\in N_0$ is a radially-symmetric continuous function with minimum at the south pole and satisfies
\begin{eqnarray*}
\mu(\{x\in M| F(s,x)<t \})=\mu(\{y\in S^m_V| F_{\bar{*}}(s,y) <t \})
\end{eqnarray*}
for all $t\in \Bbb R$. By the radial symmetry of a function on $S^m$ we mean its invariance under the standard action of $O(m)$ around the axis passing through the two poles. By Morse theory, $F_{\bar{*}}|_{\{s\}\times S^m_V}$ for $s\in N_0$ is not only continuous but also piecewise-smooth. An important property of $F_{\bar{*}}$ is that for any $p> 0$ and $(a,b)\subset \Bbb R$
\begin{eqnarray}\label{libraryinNY}
\int_{\{x\in M| a<F(s,x)<b \}}F(s,x)^p\ d\mu_g=\int_{\{y\in S^m_{V}|a<F_{\bar{*}}(s,y)<b\}} F_{\bar{*}}(s,y)^p\ d\mu_{\textsl{g}_V}
\end{eqnarray}
where $(\cdot)^p$ means $|\cdot|^p$ if $p$ is not an integer, and it extends uniquely to quite a regular function defined on the whole $N\times S^m_V$.
\begin{Theorem}[\cite{Sung}]\label{main-estimates}
Let $(M^m,g)$ and $(N^n,h)$ be smooth closed Riemannian manifolds such that $M$ has volume $V$. If $F:N\times M\stackrel{C^\infty}{\rightarrow} \Bbb R$ is a generic-fiberwise Morse function, then $F_{\bar{*}}$ extends to a unique radially-symmetric Lipschitz continuous function in $L_1^p(N\times S^m_V)$ for any $p\geq 1$ which is smooth on an open dense subset with measure-zero complement. In addition, if a Lie group $G$ acts on $(M,g)$ isometrically and on $N$ both by smooth right actions, then for any $\mathfrak{g}\in G$,
\begin{eqnarray}\label{Ocean-nuclear}
(F\cdot \mathfrak{g})_{\bar{*}}=F_{\bar{*}}\cdot \mathfrak{g}
\end{eqnarray}
where the induced $G$ action on $N\times S^m_V$ acts trivially on $S^m_V$. The analogous statement holds for a left action too.

Moreover for any $p\geq 1$, $q\in N_0(F)$, and $v\in T_qN$
$$\int_{S^m_V} \left|v(F_{\bar{*}})\right|^p d\mu_{\textsl{g}_V} \leq
\int_{M} \left|v(F)\right|^p d\mu_g, \ \ \ \
\int_{S^m_V} \left|d^S F_{\bar{*}}\right|_{\textsl{g}_V}^p d\mu_{\textsl{g}_V} \leq C^p
\int_{M} \left|d^M F\right|_g^p d\mu_g$$
%Moreover $$||d^SF_{\bar{*}}||_{L^p} \leq C ||d^M F||_{L^p}\ \ \ \textrm{and}\ \ \ ||d^NF_{\bar{*}}||_{L^p} \leq ||d^N F||_{L^p}$$ while  $||F_{\bar{*}}||_{L^p}=||F||_{L^p}$,
where $d^M$ and $d^S$ are the exterior derivatives (a.k.a. differentials) on $m$-manifolds $\{q\}\times M$ and $\{q\}\times S^m$ respectively and $C>0$ is a constant depending only on $(M,g)$. If $Ric_g\geq k>0$, then $$C=\left(\frac{V_m}{V}\right)^{\frac{1}{m}}\sqrt{\frac{m-1}{k}}.$$
\end{Theorem}
%By the uniform continuity of $F$ and $F_{\bar{*}}$ and the denseness of $N_0$, (\ref{libraryinLondon}) and (\ref{libraryinNY}) hold for any $s\in N$.
Furthermore the spherically-rearranging operation as a map between function spaces is also Lipschitz continuous.
\begin{Theorem}[\cite{Sung}]\label{Firstlady}
Suppose $(M^m,g)$ and $(N^n,h)$ are smooth closed Riemannian manifolds such that $M$ has volume $V$ and $Ric_g>0$.
Let $$\mathcal{F}:=\{F:N\times M\stackrel{C^\infty}{\rightarrow} \Bbb R|F \textrm{ is a generic-fiberwise Morse function.}\}$$ and $$\tilde{\mathcal{F}}:=\{F_{\bar{*}}:N\times S^m_V\rightarrow \Bbb R|F\in \mathcal{F}\}$$ be endowed with $C^0$-topology.
Then the spherically-rearranging map $\Phi:\mathcal{F}\rightarrow \tilde{\mathcal{F}}$  given by $\Phi(F)=F_{\bar{*}}$ is Lipschitz continuous with Lipschitz constant 1,
i.e. for any $F, G\in \mathcal{F}$   $$||G_{\bar{*}}-F_{\bar{*}}||_\infty\leq ||G-F||_\infty.$$
\end{Theorem}
In fact the above theorem holds without the condition $Ric_g>0$, since the condition was not used in the proof in \cite{Sung}. Anyway we shall use the above theorem only for the case of $Ric_g>0$.

Parallel stories work well with Euclidean and hyperbolic symmetrizations. For a smooth generic-fiberwise Morse function $F :N\times M \rightarrow \Bbb R$,  $$F_{\bar{\star}}:N_0\times D^m_V\rightarrow \Bbb R$$ is defined as the fiberwise Euclidean rearrangement of $F$; namely, for all $s\in N_0$ $F_{\bar{\star}}|_{\{s\}\times D^m_V}$ is a radially-symmetric continuous function with minimum at the origin and satisfies
\begin{eqnarray*}
\mu(\{p\in M| F(s,p)<t \})=\mu(\{q\in D^m_V| F_{\bar{\star}}(s,q) <t \})
\end{eqnarray*}
for all $t\in \Bbb R$. Then in the same way as
%the integral equality for $F_{\bar{\star}}$ and $F$ corresponding to
(\ref{libraryinNY}) we have
\begin{eqnarray*}
\int_{\{x\in M| a<F(s,x)<b \}}F(s,x)^p\ d\mu_g=\int_{\{y\in D^m_{V}|a<F_{\bar{\star}}(s,y)<b\}} F_{\bar{\star}}(s,y)^p\ d\mu_{g_{_{\Bbb E}}}
\end{eqnarray*}
for any $s\in N_0$, and also
\begin{Theorem}[\cite{Sung}]\label{Young&Hyuk}
Let $(M^m,g)$ with volume $V$ and $(N^n,h)$ be smooth closed Riemannian manifolds and $F:N\times M \rightarrow \Bbb R$ be a smooth generic-fiberwise Morse function.
Then $F_{\bar{\star}}$ extends to a unique radially-symmetric Lipschitz continuous $L_1^p$ function in $N\times \overline{D^m_V}$ for any $p\geq 1$ which is smooth on an open dense subset with measure-zero complement.

If an open subset $\mathfrak{B}\subset (M,g)$ is isometric to a bounded domain of $(\Bbb R^m, g_{_{\Bbb E}})$,
and $F^{-1}(-\infty,b]$ for $b\in \Bbb R$ is contained in $N\times \mathfrak{B}$, then for any $a<b$
\begin{eqnarray*}
\int_{F^{-1}_{\bar{\star}}(a,b]} |dF_{\bar{\star}}|^p\ d\mu_{h+g_{_{\Bbb E}}}
\leq
\int_{F^{-1}(a,b]} |dF|^p\ d\mu_{h+g}.
\end{eqnarray*}
%while $||F_{\bar{\star}}||_{L^p}=||F||_{L^p}$.
\end{Theorem}
In case of hyperbolic symmetrization, one has to use the Poincar\'{e} $m$-disk $(\Bbb H^m,g_{_{\Bbb H}})$ and the open $m$-disk $\mathfrak{D}^m_V\subset \Bbb H^m$ with center at the origin and hyperbolic volume $V$. Then the  statement analogous to the above theorem still holds for the fiberwise hyperbolic arrangement of $F$. (Of course, now $\mathfrak{B}\subset (M,g)$ is supposed to be isometric to a bounded domain in $(\Bbb H^m,g_{_{\Bbb H}})$.)

\section{$\lambda_i$ comparison on fiber bundles}

In this section we introduce a ``canonical" $G$-invariant differential form $\Omega_G$ of degree $\dim G$ on a principal $G$-bundle $P$ with a connection and a bi-invariant metric on $G$. It acts as a fiberwise volume form on $P$ and will play a pivotal role in computing the Rayleigh quotient of a fiber bundle associated to $P$.

\subsection{Introducing $\Omega_G$ on a principal $G$-bundle}
The purpose of introducing $\Omega_G$ is in proving Lemma \ref{Woody&Jessie} which is the main goal of this subsection.
Let's start with a quick review of a Riemannian submersion.
Let $$\pi: (X,g_{_X})\rightarrow (B,g_{_B})$$ be a Riemannian submersion with totally geodesic fibers and $H$ be its horizontal distribution. We assume that $B$ and fiber $F$ are smooth closed manifolds of positive dimensions $\dim B=n$ and $\dim F=m$ respectively.

It is well known (\cite{Besse,Hermann}) that all fibers are diffeomorphic via the ``parallel transport" induced by $H$ so that $X$ is an $F$-fiber bundle with a structure group $G$ given by the ``holonomy" group of the parallel transport. Since all the fibers are totally geodesic, $G$ is a compact Lie group consisting of isometries of $(F,g_{_F})$. Thus there exists a smooth principal $G$-bundle $P$ with $$\pi_{_P}: P\rightarrow B$$ such that $X$ is the associated bundle of $P$ with $H$ being induced by the horizontal distribution $\mathcal{H}$ of a $G$-connection on $P$.
%(Explicitly $P$ is given by $\bigcup_{b\in B}G_b$ where $G_b$ is the set of isometries from $\pi^{-1}(b_0)$ to  $\pi^{-1}(b)$ obtained by parallel transports via piecewise-smooth continuous curves from a fixed reference point $b_0\in B$ to $b$ in $B$ (cf. \cite{Hermann}).)
Namely $X$ is diffeomorphic to the quotient of $P\times F$ by the free right $G$-action given by $$(p,f)\cdot g=(p\cdot g,g^{-1}\cdot f)$$ for $g\in G$. We denote this quotient map by $$\tilde{\pi}:P\times F\rightarrow X.$$

Using a local trivialization of $P$, $\tilde{\pi}:U_\alpha\times G\times F\rightarrow U_\alpha\times F$ is expressed as
\begin{eqnarray}\label{yatap}
\tilde{\pi}(x,g,f)=(x,g\cdot f).
\end{eqnarray}
For another trivialization over $U_\beta\subset B$ there exists a smooth transition function $g_{\alpha\beta}:U_\alpha\cap U_\beta\rightarrow G$ identifying $$(x,g,f)\in U_\alpha\times G\times F\ \ \ \ \textrm{and}\ \ \ \ (x,g_{\alpha\beta}(x)g,f)\in U_\beta\times G\times F,$$ and it projects down to the identification of $$(x,g\cdot f)\in U_\alpha\times F  \ \ \ \ \textrm{and}\ \ \ \   (x,g_{\alpha\beta}(x)\cdot (g\cdot f))\in U_\beta\times F.$$
In this local trivialization $U_\alpha\times F$ of $X$ associated to a local trivialization of $P$, a horizontal curve of $X$ is written as
\begin{eqnarray}\label{yatapgo}
\{(x(t),g(t)\cdot f)|g(t)\in G, t\in [0,1]\}
\end{eqnarray}
which induces a horizontal curve $\{(x(t),g(t))|t\in [0,1]\}$ in $U_\alpha\times G$, thereby producing a $G$-invariant (horizontal) distribution $\mathcal{H}$, i.e. a connection in $P$.

Let's see how the metric $g_{_X}$ is expressed in the above local trivialization $U_\alpha\times F$.
Since the parallel transport of $\{x(0)\}\times F$ to  $\{x(1)\}\times F$ given by $$f\mapsto g(1)\cdot f$$ is isometric, the metric restricted to $\{x(1)\}\times F$ must be equal to the metric restricted to $\{x(0)\}\times F$. Thus let us express the metric $g_{_X}$ as
\begin{eqnarray}\label{jackpot}
g_{_X}=\pi^*(g_{_B})|_H\oplus g_{_F}|_{H^\perp}
\end{eqnarray}
by which we mean that in an appropriate local trivialization $TF$ is endowed with metric $g_{_F}$ and its orthogonal complement $H$ is endowed with metric $\pi^*(g_{_B})$. We shall use this notation to express such a metric.

In fact any smooth principal $G$-bundle $P$ with a $G$-connection and a $G$-invariant metric $g_{_F}$ on $F$ give rise to a Riemannian submersion metric (\ref{jackpot}) on the associated bundle via the above correspondence of $H$ and $\mathcal{H}$ such that all of its fibers are totally geodesic.(\cite{Besse, Vilms})

Note that $P$ as a $G$-fiber bundle is the associated bundle of a principal bundle $P$ itself. So it can be endowed with a Riemannian submersion over $(B,g_{_B})$ with horizontal distribution induced by the horizontal distribution $\mathcal{H}$ as a principal $G$-connection of $P$. In fact those two distributions coincide, because the horizontal distribution given by curves $(x,g(t)\cdot g)$ as in (\ref{yatapgo}) is just $\mathcal{H}$ which is invariant under the right $G$ action.
When $G$ has positive dimension, we put a bi-invariant metric $g_{_G}$ on $G$, and the metric $$g_{_P}:=\pi^*(g_{_B})|_{\mathcal{H}}\oplus g_{_G}|_{\mathcal{H}^\perp}$$ on $P$ given as in (\ref{jackpot}) makes $\pi_{_P}$ a Riemannian submersion with totally geodesic fibers over $(B,g_{_B})$. By the right $G$-invariance of $\mathcal{H}$ and $g_{_G}$, $g_{_P}$ is invariant under the right $G$-action. If $G$ is finite, then $\pi_{_P}$ is a covering projection and $g_{_P}$ is defined as the pull-back metric $\pi_{_P}^*(g_{_B})$.

Henceforth we put $g_{_P}$ on $P$ unless otherwise specified, and set $\ell:=\dim (G)$. Choose any orientation of $G$, and then the following $\ell$-form on $P$ shall play an important role in proving Theorem \ref{Main-theorem}.
\begin{Proposition}
There exists a smooth $G$-invariant $\ell$-form $\Omega_G$ on $P$ such that $\Omega_G$ restricted to each fiber of $\pi_{_P}$ is the volume form of the fiber and the interior product $\iota_v\Omega_G$ is 0 for any horizontal vector $v$ on $P$.
\end{Proposition}
\begin{proof}
If $\ell=0$, then the 0-form $\Omega_G$ is 1 and there is nothing to prove.

We now assume that $\ell>0$, and  first need to prove the volume form on $G$ can be extended to a global form on $P$.
The bi-invariant metric $g_{_G}$ of $G$ gives an $Ad(G)$-invariant inner product on the Lie algebra $\mathcal{G}:=T_eG$, which in turn induces an $Ad(G)$-invariant volume form $\mathfrak{F}$ for $\mathcal{G}$. Let $e_1,\cdots,e_\ell$ be an orthonormal basis of $\mathcal{G}$ so that $\mathfrak{F}(e_1,\cdots,e_\ell)=1$,
and $\vartheta:TP\rightarrow \mathcal{G}$ be the connection 1-form on $P$ corresponding to $\mathcal{H}$. Then we define $$\Omega_G:=\textrm{Alt}(\mathfrak{F}(\vartheta,\cdots,\vartheta))$$ which is the alternation\footnote{The alternation $\textrm{Alt}(T)$ of a multi-linear map $T:\prod_{i=1}^\ell V\rightarrow \Bbb R$ is defined as $$\textrm{Alt}(T)(v_1,\cdots,v_\ell)=
\frac{1}{\ell !}\sum_{\sigma\in S_\ell}sgn(\sigma)T(v_{\sigma(1)},\cdots,v_{\sigma(\ell)})
$$ for any $v_1,\cdots,v_\ell\in V$.} of a multi-linear map $$\mathfrak{F}(\vartheta,\cdots,\vartheta):\prod_{i=1}^\ell TP\rightarrow \Bbb R.$$
To see more clearly that the above-defined $\Omega_G$ is in $\Lambda^{\ell}P$, writing $\vartheta$ as $\sum_{i=1}^\ell\vartheta_ie_i$ for $\vartheta_i\in \Lambda^1(P)$ and plugging it into the above definition, we have that
\begin{eqnarray*}
\Omega_G &=&\sum_{\sigma\in S_\ell}\textrm{Alt}(\vartheta_{\sigma(1)}\otimes\cdots\otimes \vartheta_{\sigma(\ell)})\mathfrak{F}(e_{\sigma(1)},\cdots,e_{\sigma(\ell)})\\
&=&\frac{1}{\ell !}\sum_{\sigma\in S_\ell}\vartheta_{\sigma(1)}\wedge\cdots\wedge\vartheta_{\sigma(\ell)}\mathfrak{F}(e_{\sigma(1)},\cdots,e_{\sigma(\ell)})\\
&=& \frac{1}{\ell !}\sum_{\sigma\in S_\ell}\vartheta_1\wedge\cdots\wedge\vartheta_\ell\  \mathfrak{F}(e_1,\cdots,e_\ell)\\ &=& \vartheta_1\wedge\cdots\wedge\vartheta_\ell.
\end{eqnarray*}
%where we used the fact that $\textrm{Alt}(\vartheta_{i_1}\otimes\cdots\otimes \vartheta_{i_\ell})=\frac{1}{\ell !}\vartheta_{i_1}\wedge\cdots\wedge \vartheta_{i_\ell}.$

Its invariance under the $G$-action follows from
\begin{eqnarray*}
R_g^*(\Omega_G)&=&\textrm{Alt}(\mathfrak{F}(R_g^*\vartheta,\cdots,R_g^*\vartheta))\\ &=&\textrm{Alt}(\mathfrak{F}(Ad(g^{-1})\cdot\vartheta,\cdots,Ad(g^{-1})\cdot\vartheta))\\ &=& \textrm{Alt}(\mathfrak{F}(\vartheta,\cdots,\vartheta))\\ &=& \Omega_G,
\end{eqnarray*}
and $\iota_v\Omega_G=0$ holds for any $v\in \mathcal{H}$ because $\imath_v\vartheta=0$ for horizontal $v$.

Recall that the fundamental vector fields $e_1^*,\cdots,e_\ell^*$ on $P$ defined by $$e_i^*(p):=\frac{d}{dt}|_{t=0}(p\cdot exp(te_i))\in T_pP$$ for $p\in P$ are the left-invariant extensions of $e_1,\cdots,e_\ell\in T_eG$ in a local trivialization. So they are orthonormal at each point, and
\begin{eqnarray*}
\Omega_G(e_1^*,\cdots,e_\ell^*)&=& \frac{1}{\ell !}\sum_{\sigma\in S_\ell}sgn(\sigma)\mathfrak{F}(\vartheta(e^*_{\sigma(1)}),\cdots,\vartheta(e^*_{\sigma(\ell)}))\\ &=& \frac{1}{\ell !}\sum_{\sigma\in S_\ell}sgn(\sigma)\mathfrak{F}(e_{\sigma(1)},\cdots,e_{\sigma(\ell)})\\&=& \mathfrak{F}(e_{1},\cdots,e_{\ell})\\ &=& 1
\end{eqnarray*}
proving that $\Omega_G$ gives the volume form on each fiber of $P$. \footnote{We remark that $\Omega_G$ is not closed for a general $\mathcal{H}$, but obviously closed if $[\mathcal{H},\mathcal{H}]\subset \mathcal{H}$, i.e. the connection is flat.}
\end{proof}

By abuse of notation, $\mathcal{H}$ and $\mathfrak{D}(G):=\mathcal{H}^\perp$ on $(P,g_{_P})$ also denote obvious lifted distributions on $P\times F$ via the standard projection map $p_1:P\times F\rightarrow P$, and we put the product metric $g_{_P}+g_{_F}$ on $P\times F$. Despite the $G$-invariance of $g_{_P}+g_{_F}$,
$$\tilde{\pi}: (P\times F,g_{_P}+g_{_F})\rightarrow (X,g_{_X})$$ is not a Riemannian submersion in general as seen in :
\begin{Lemma}\label{tram}
Let $\mathfrak{D}(F)$  and $\tilde{\mathfrak{D}}(G)$ be smooth distributions of $P\times F$ consisting of tangent vectors of each $\{pt\}\times F$ and orbits of the $G$-action on $P\times F$ respectively.
Then the $G$-invariant orthogonal direct sum $\mathcal{H}\oplus \mathfrak{D}(F)$ is orthogonal to $\mathfrak{D}(G)$, but not orthogonal to $\tilde{\mathfrak{D}}(G)$ in general, and the differential $$d\tilde{\pi}: \mathcal{H}\oplus \mathfrak{D}(F)\rightarrow TX$$  at any point of $P\times F$ is an isometric isomorphism.

For any differential form $\alpha$ on $X$,
\begin{eqnarray}\label{shatter}
|(\tilde{\pi}^*\alpha)\wedge \Omega_G|_{g_{_P}+g_{_F}}=\tilde{\pi}^*|\alpha|_{g_{_X}}.
\end{eqnarray}
\end{Lemma}
\begin{proof}
The $G$-invariance of $\mathcal{H}\oplus \mathfrak{D}(F)$ and the orthogonalities
$$\mathcal{H}\perp \mathfrak{D}(F),\ \ \ \ \ \ \ \ (\mathcal{H}\oplus \mathfrak{D}(F))\perp \mathfrak{D}(G)$$ are obvious.
So $\mathcal{H}\oplus \mathfrak{D}(F)$ is not orthogonal to $\tilde{\mathfrak{D}}(G)$ unless all $G$-orbits of $F$ are zero-dimensional.

To prove the remaining parts, let's see it in the local trivialization (\ref{yatap}).
A local smooth curve $c(t):=(x(t),g(t),f)\in U_\alpha\times G\times F$ tangent to $\mathcal{H}$ gets mapped  by $\tilde{\pi}$ to $\tilde{\pi}(c(t))=(x(t),g(t)\cdot f)$ which is a horizontal curve of $X$, and both curves project down to a curve $x(t)$ in $B$
by $\pi_{_P}\circ p_1$ and $\pi$ respectively.
So
\begin{eqnarray*}
|c'(t)|_{g_{_P}+g_{_F}}&=&|(x'(t),g'(t))|_{g_{_P}}\\ &=&|x'(t)|_{g_{_B}}\\ &=&|\frac{d}{dt}\tilde{\pi}(c(t))|_{g_{_X}}
\end{eqnarray*}
where the 2nd equality holds because $(x'(t),g'(t))$ belongs to the horizontal distribution $\mathcal{H}$ of Riemannian submersion metric $g_{_P}$.

A local smooth curve $\gamma(t):=(x,g,f(t))\in U_\alpha\times G\times F$ tangent to $\mathfrak{D}(F)$ gets mapped to $\tilde{\pi}(\gamma(t))=(x,g\cdot f(t))$ which is a vertical curve, i.e. staying in a fiber of $X$, and
\begin{eqnarray*}
|\gamma'(t)|_{g_{_P}+g_{_F}}&=&|f'(t)|_{g_{_F}}\\ &=&|\frac{d}{dt}(g\cdot f(t))|_{g_{_F}}\\ &=& |\frac{d}{dt}\tilde{\pi}(\gamma(t))|_{g_{_X}},
\end{eqnarray*}
where the 2nd equality holds because the $G$-action on $F$ is isometric.

Since $d\tilde{\pi}(c'(t))$ is horizontal and  $d\tilde{\pi}(\gamma'(t))$ is vertical,
$$\langle d\tilde{\pi}(c'(t)),d\tilde{\pi}(\gamma'(t))\rangle = 0 = \langle c'(t),\gamma'(t)\rangle.$$
Now we have that $\mathcal{H}\oplus \mathfrak{D}(F)$ orthogonal to $\mathfrak{D}(G)$ gets mapped to $TX$ isometrically by $d\tilde{\pi}$. By combining it with the fact that $\Omega_G$ restricted to each $G$ fiber is the volume form such that $\iota_v\Omega_G$ is 0 for any horizontal vector $v$ on $P$, (\ref{shatter}) follows immediately.
\end{proof}

By abuse of notation, $\Omega_G$ on $P\times F$ shall mean $p_1^*\Omega_G$.
\begin{Lemma}\label{Chokuk4}
Suppose $B, F$, and $X$ are orientable, and put the product orientations on $X$, $P$ and $P\times F$ induced from orientations of $B, F$, and $G$.
Let $\omega_{P\times F}$ and $\omega_X$ be the volume forms of $g_{_P}+g_{_F}$ and $g_{_X}$ respectively. Then
$$\omega_{P\times F}=(-1)^{\ell m}\tilde{\pi}^*(\omega_X)\wedge  \Omega_G.$$
\end{Lemma}
\begin{proof}
%First we note that $\omega_F$ is not globally defined in $X$. For a different local trivialization $U_\beta\times F$ and  $x\in U_\alpha\cap U_\beta$, $(x,f)\in U_\alpha\times F$ is identified with $(x,g_{\alpha\beta}(x)\cdot f)\in U_\beta\times F$ and so $\omega_F$ pulls back to $\omega_F$ plus certain forms involving $dg_{\alpha\beta}(x)$. Thus $\pi^*(\omega_B)\wedge \omega_F$ is invariant under the gluing map to give a global form.
Let $\{v_1,\cdots,v_n\}$, $\{v_{n+1},\cdots,v_{n+\ell}\}$, and $\{v_{n+\ell+ 1},\cdots,v_{n+\ell+m}\}$ be positively-oriented local orthonormal frames of
$\mathcal{H}$, $\mathfrak{D}(G)$, and $\mathfrak{D}(F)$ respectively. By the product orientation of $P$, we mean that $\{v_1,\cdots,v_{n+\ell}\}$
%\ \ \ \textrm{and}\ \ \
is a positively-oriented local orthonormal frame of $P$.
%$\{v_1,\cdots,v_{n+\ell+m}\}$
Similarly $$\{d\tilde{\pi}(v_1),\cdots,d\tilde{\pi}(v_n),d\tilde{\pi}(v_{n+\ell+1}),\cdots,d\tilde{\pi}(v_{n+\ell+m})\}$$ gives a positively-oriented local orthonormal frame on $X$ by the above lemma.
Using $\iota(v_j)\Omega_G=0$ for any $v_j\notin \{v_{n+1}\cdots,v_{n+\ell}\}$,
\begin{eqnarray*}
 & & (\tilde{\pi}^*(\omega_X)\wedge\Omega_G)(v_1,\cdots,v_{n+\ell+m})\\
&=& (-1)^{\ell m}(\tilde{\pi}^*(\omega_X)\wedge\Omega_G)(v_1,\cdots,v_n,v_{n+\ell+1},\cdots, v_{n+\ell+m},v_{n+1}\cdots,v_{n+\ell})\\
&=&(-1)^{\ell m} \omega_X(d\tilde{\pi}(v_1),\cdots,d\tilde{\pi}(v_n),d\tilde{\pi}(v_{n+\ell+1}),\cdots,d\tilde{\pi}(v_{n+\ell+m}))\ \Omega_G(v_{n+1}\cdots,v_{n+\ell})\\
&=& (-1)^{\ell m}
\end{eqnarray*}
which implies the desired equality.
%To prove the 2nd equality, consider $$d\pi_{_P}:\mathcal{H}\oplus \mathfrak{D}(G)\rightarrow TB.$$ Since $d\pi_{_P}|_{\mathcal{H}}$ at each point of $P$ is isometric onto $TB$, $d\pi_{_P}|_{\mathfrak{D}(G)}=0$, and $\mathcal{H}\perp \mathfrak{D}(G)$, $*_{_P}\pi_{_P}^*(\omega_B)=\Omega_G$ must follow.

%The 3rd equality is proved likewise.  Consider $$d\tilde{\pi}:\mathcal{H}\oplus \mathfrak{D}(F)\oplus \tilde{\mathfrak{D}}(G)\rightarrow TX.$$  From Lemma \ref{tram}, we know that $d\tilde{\pi}|_{\mathcal{H}\oplus \mathfrak{D}(F)}$ at each point $z\in P\times F$ is isometric onto $T_{\tilde{\pi}(z)}X$, and $(\mathcal{H}\oplus \mathfrak{D}(F))\perp \mathfrak{D}(G)$. By $d\tilde{\pi}|_{\tilde{\mathfrak{D}}(G)}=0$, $\imath_v \tilde{\pi}^*\beta=0$ for any form $\beta$ on $X$ and any tangent vector $v$ belonging to $\tilde{\mathfrak{D}}(G)$. So $*_{_{P\times F}}\tilde{\pi}^*(\omega_X)=(-1)^{\ell m}\ \Omega_G$.
\end{proof}

The following is a key lemma.
\begin{Lemma}\label{Woody&Jessie}
For any $\varphi\in L_1^2(X)$, $$\int_{P\times F} \tilde{\pi}^*\varphi\ d\mu_{g_{_P}+g_{_F}}=|G|\int_X \varphi\ d\mu_{g_{_X}}$$
$$\int_{P\times F} |d(\tilde{\pi}^*\varphi)\wedge \Omega_G|^2_{g_{_P}+g_{_F}} d\mu_{g_{_P}+g_{_F}}=|G|\int_X |d\varphi|^2_{g_{_X}}\ d\mu_{g_{_X}}.$$
%where $|\cdot|_{_X}$ and $|\cdot|_{_{P\times F}}$ denote norms w.r.t. $g_{_X}$ and $g_{_P}+g_{_F}$ respectively.
\end{Lemma}
\begin{proof}
%Note that if $\varphi_n\rightarrow \varphi$ in $L_1^2(X)$ then $\tilde{\pi}^*\varphi_n\rightarrow \tilde{\pi}^*\varphi$ in $L_1^2(P\times F)$. Thus it's enough to prove the above equalities for smooth $\varphi$. (For the 2nd equality, also by using Lemma \ref{jeonghoon}.)
To prove the 1st equality, it suffices to prove
$$\int_{\tilde{\pi}^{-1}(\mathcal{U})} \tilde{\pi}^*(\varphi)\ d\mu_{g_{_P}+g_{_F}}=|G|\int_{\mathcal{U}} \varphi\ d\mu_{g_{_X}}$$
for any open neighborhood $\mathcal{U}\subset X$ so small that it is orientable and $\tilde{\pi}^{-1}(\mathcal{U})$ is diffeomorphic to $\mathcal{U}\times G$.
Here we don't require the global orientability of $X$ and $P\times F$; however we take orientations of $\mathcal{U}$ and $\tilde{\pi}^{-1}(\mathcal{U})$ so that $\omega_{_{\mathcal{U}}}$ and $\tilde{\pi}^*(\omega_{_{\mathcal{U}}})\wedge\Omega_G$ are the local volume forms respectively. Recall that $|G|$ denotes the volume of $(G,g_{_G})$ when $G$ is an infinite group.
\begin{Sublemma}\label{gotica}
Let $\mathfrak{G}$ be any orbit of the $G$-action on $P\times F$ and $i:\mathfrak{G}\hookrightarrow P\times F$ be the inclusion map. Then $$\int_{\mathfrak{G}}i^*\Omega_G=|G|.$$
\end{Sublemma}
\begin{proof}
Under $p_1:P\times F\rightarrow P$, $\mathfrak{G}$ gets mapped diffeomorphically onto $p_1(\mathfrak{G})$ which is a fiber of the principal $G$-bundle $P$. So
\begin{eqnarray*}
\int_{\mathfrak{G}}i^*\Omega_G &=&\int_{\mathfrak{G}}i^*p_1^*\Omega_G \\ &=&\int_{\mathfrak{G}}(p_1\circ i)^*\Omega_G\\ &=& \int_{p_1(\mathfrak{G})}\Omega_G\\ &=& |G|.
\end{eqnarray*}
\end{proof}
Using that $\tilde{\pi}^{-1}(\mathcal{U})\simeq\mathcal{U}\times G$ and $\omega_{_{\mathcal{U}}}$ is a nowhere-vanishing $(n+m)$-form for $n+m=\dim X$,
%$s(\mathcal{U})\times \mathfrak{G}$ where $\mathfrak{G}$ is a $G$-orbit in $\tilde{\pi}^{-1}(\mathcal{U})$ and $s:\mathcal{U}\rightarrow \tilde{\pi}^{-1}(\mathcal{U})$ is a local section which is transversal to $\mathfrak{G}$, and so
\begin{eqnarray*}
\int_{\tilde{\pi}^{-1}(\mathcal{U})} \tilde{\pi}^*(\varphi)\ d\mu_{g_{_P}+g_{_F}}&=&
\int_{\tilde{\pi}^{-1}(\mathcal{U})} \tilde{\pi}^*(\varphi)\ \tilde{\pi}^*(\omega_{_{\mathcal{U}}})\wedge\Omega_G \\ &=&
%\int_{s(\mathcal{U})\times \mathfrak{G}} \tilde{\pi}^*(\varphi\ \omega_{_{\mathcal{U}}})\wedge\Omega_G \\ &=& \int_{s(\mathcal{U})\times \mathfrak{G}} \tilde{\pi}^*(\varphi\ \omega_{_{\mathcal{U}}})\wedge i^*\Omega_G \\ &=&
\int_{\mathcal{U}}(\varphi\ \omega_{_{\mathcal{U}}}\int_{\mathfrak{G}} i^*\Omega_G) \\ &=& \int_{\mathcal{U}}\varphi\ d\mu_{g_{_X}}|G|
\end{eqnarray*}
where $i$ is the inclusion map of each $\mathfrak{G}=\{\textrm{pt}\}\times G\subset \mathcal{U}\times G$ and we performed the integration along fibers $\mathfrak{G}$ using Sublemma \ref{gotica}. So the proof of the 1st equality is done.

To prove the 2nd equality we use the already-proven 1st equality and (\ref{shatter}), and get
\begin{eqnarray*}
|G|\int_X|d\varphi|_{g_{_X}}^2\ d\mu_{g_{_X}} &=& \int_{P\times F} \tilde{\pi}^*(|d\varphi|_{g_{_X}}^2)\ d\mu_{g_{_P}+g_{_F}}\\&=& \int_{P\times F} |\tilde{\pi}^*(d\varphi)\wedge \Omega_G|_{g_{_P}+g_{_F}}^2\ d\mu_{g_{_P}+g_{_F}}\\&=&
\int_{P\times F} |d(\tilde{\pi}^*\varphi)\wedge \Omega_G|_{g_{_P}+g_{_F}}^2\ d\mu_{g_{_P}+g_{_F}}.
\end{eqnarray*}

\end{proof}

\subsection{Warped variation of Riemannian submersion}\label{Korean-nuclear}

There is a so-called \textit{canonical variation} of a Riemannian submersion which only deforms the scale of all fibers by the same conformal factor
while keeping anything else such as the horizontal distribution.
Generalizing it, we define a \textit{warped variation} of a Riemannian submersion on $(X, g_{_X})$ by a deformation of the scales of fibers where the rescaling factors of fibers may not be the same.
More precisely, for $\rho:=\check{\rho}\circ \pi$ where $\check{\rho}$ is a smooth positive function on $B$, a warped variation $g_{_X}^\rho$ of $g_{_X}$ is defined as $$g_{_X}^\rho|_{H^\perp}=\rho^2g_{_X}|_{H^\perp},\ \ \ \ g_{_X}^\rho|_{H}=g_{_X}|_{H},\ \ \ \ g_{_X}^\rho(v,w)=0$$ for any $v\in H^\perp$ and $w\in H$ where $H^\perp$ is the distribution given by fibers.

Since $\pi\circ \tilde{\pi}=\pi_{_P}\circ p_1$ for $p_1:P\times F\rightarrow P$,
$$\tilde{\pi}^*\rho=\check{\rho}\circ\pi\circ \tilde{\pi}=\check{\rho}\circ\pi_{_P}\circ p_1=p_1^*(\pi_{_P}^*\check{\rho}),$$
and hence $\tilde{\pi}^*\rho$ can be viewed as a function on $P$, which we denote shortly by $\tilde{\rho}$.
Obviously the warped product metric $g_{_P}+\tilde{\rho}^2g_{_F}$ is still $G$-invariant.
We also denote the norms of (co)tangent vectors of $X$ w.r.t. $g_{_X}^\rho$ by $|\cdot|_\rho$, and norms on $P\times F$ w.r.t. $g_{_P}+\tilde{\rho}^2g_{_F}$ by $|\cdot|_{\tilde{\rho}}$.

Lemma \ref{Woody&Jessie} is generalized to the following proposition which is the exact forms of equalities we shall use to prove Theorem \ref{Main-theorem}.
\begin{Proposition}\label{Matt-Damon}
For any $\varphi\in L_1^2(X)$, $$\int_{P\times F} \tilde{\pi}^*\varphi\ d\mu_{g_{_P}+\tilde{\rho}^2g_{_F}}=|G|\int_X \varphi\ d\mu_{g_{_X}^\rho}$$
$$\int_{P\times F} |d(\tilde{\pi}^*\varphi)\wedge \Omega_G|^2_{\tilde{\rho}}\ d\mu_{g_{_P}+\tilde{\rho}^2g_{_F}}=|G|\int_X |d\varphi|^2_\rho\ d\mu_{g_{_X}^\rho}.$$
\end{Proposition}
\begin{proof}
%As explained in Lemma \ref{Woody&Jessie}, it's enough to prove for a smooth $\varphi$.
The 1st equality immediately follows from Lemma \ref{Woody&Jessie}.
\begin{eqnarray*}
\int_{P\times F} \tilde{\pi}^*\varphi\ d\mu_{g_{_P}+\tilde{\rho}^2g_{_F}}&=&
\int_{P\times F} \tilde{\pi}^*(\varphi)\ \tilde{\rho}^m\ d\mu_{g_{_P}+g_{_F}}\\&=&
\int_{P\times F} \tilde{\pi}^*(\varphi\rho^m)\ d\mu_{g_{_P}+g_{_F}}\\&=&
|G|\int_X \varphi\rho^m\ d\mu_{g_{_X}}\\ &=& |G|\int_X \varphi\ d\mu_{g_{_X}^\rho}.
\end{eqnarray*}

For the 2nd equality, first note that Lemma \ref{tram} also holds for $(P\times F,g_{_P}+\tilde{\rho}^2g_{_F})$ and $(X,g_{_X}^\rho)$.
Namely, the orthogonal direct sum $\mathcal{H}\oplus \mathfrak{D}(F)$ orthogonal to $\mathfrak{D}(G)$ gets mapped isometrically onto $TX$ by $d\tilde{\pi}$, and
\begin{eqnarray*}
\tilde{\pi}^*(|\alpha|_\rho)= |(\tilde{\pi}^*\alpha)\wedge \Omega_G|_{\tilde{\rho}}.
\end{eqnarray*}
for any differential form $\alpha$ on $X$.

Thus
\begin{eqnarray*}
\int_{P\times F} |d(\tilde{\pi}^*\varphi)\wedge \Omega_G|^2_{\tilde{\rho}}\ d\mu_{g_{_P}+\tilde{\rho}^2g_{_F}} &=&
\int_{P\times F} \tilde{\pi}^*(|d\varphi|^2_\rho)\ d\mu_{g_{_P}+\tilde{\rho}^2g_{_F}} \\ &=&
|G|\int_X |d\varphi|^2_\rho\ d\mu_{g_{_X}^\rho}
\end{eqnarray*}
by using the already-proven 1st equality.

\end{proof}

\subsection{Proof of Theorem \ref{Main-theorem}}

The basic idea of proof is to use fiberwise spherical rearrangement of a first eigenfunction of $X$ following Lichnerowicz theorem. But our fiberwise symmetrization method was developed only for Riemannian products. To circumvent this difficult we shall proceed as follows.

Instead of trying to compare the Rayleigh quotients on $(X,g_{_X}^\rho)$  and  $(B\times S^m_V,g_{_B}+\check{\rho}^2g_{_{\Bbb S}})$ directly, we first lift the 1st eigenfunction of $(X,g_{_X}^\rho)$ to $(P\times F,g_{_P}+g_{_F})$ where $P$ is the principal $G$-bundle associated to $(X,g_{_X})$ as explained so far. Then take its fiberwise spherical rearrangement defined on $P\times S^m_V$, and finally project it down to $B\times S^m_V$ by averaging along $G$ fibers of $P$. The Rayleigh quotient of the last function is to be compared with that of the 1st eigenfunction we started with.

For brevity of notation let's set constants $$V:=\int_F d\mu_{g_{_F}},\ \ \ \  c:=\left(\frac{V}{V_m}\right)^{\frac{1}{m}},\ \ \ \
\tilde{C}:=(\int_{P\times S^m} d\mu_{c^2g_{_P}+\tilde{\rho}^2\textsl{g}_V})^{\frac{1}{2}}$$(by the Bishop-Gromov inequality, $c\leq 1$), and we shall continue to use notations from the previous subsections such as $\tilde{\rho}, \Omega_G,$ etc.
Let $\epsilon\ll 1$ be any positive number satisfying $$\epsilon<\frac{\sqrt{c^{n+\ell}|G|}}{2(\sqrt{c^{n+\ell}}+\tilde{C})}.$$

To begin the above procedure, let $\varphi:X\rightarrow \Bbb R$ be the first eigenfunction of $(X,g_{_X}^\rho)$ normalized as $$\int_X\varphi^2\ d\mu_{g_{_X}^\rho}=1.$$
Since $\tilde{\pi}^*\varphi$ is invariant under the right $G$-action on $P\times F$, by Theorem \ref{morse-1}
there exists a generic-fiberwise Morse function $\tilde{\varphi}$ on $P\times F$ which is $C^2$-close enough to $\tilde{\pi}^*\varphi$ so that
\begin{eqnarray}\label{body-landing-1}
\sup_{g\in G}||\tilde{\pi}^*\varphi-\tilde{\varphi}\cdot g||_{\infty}< \frac{\epsilon}{2}
\end{eqnarray}
\begin{eqnarray}\label{body-landing-2}
\left(\int_{P\times F}(\tilde{\varphi}-\tilde{\pi}^*\varphi)^2\ d\mu_{g_{_P}+\tilde{\rho}^2g_{_F}}\right)^{\frac{1}{2}}<\epsilon
%\left(\int_{P\times F}\tilde{\pi}^*\varphi^2\ d\mu_{g_{_P}+\tilde{\rho}^2g_{_F}}\right)^{\frac{1}{2}}|
\end{eqnarray}
\begin{eqnarray}\label{body-landing}
|\frac{\int_{P\times F} |d(\tilde{\pi}^*\varphi)\wedge \Omega_G|^2_{\tilde{\rho}}\ d\mu_{g_{_P}+\tilde{\rho}^2g_{_F}}}
{\int_{P\times F} (\tilde{\pi}^*\varphi)^2\ d\mu_{g_{_P}+\tilde{\rho}^2g_{_F}}}
-\frac{\int_{P\times F} |d\tilde{\varphi}\wedge \Omega_G|^2_{\tilde{\rho}}\ d\mu_{g_{_P}+\tilde{\rho}^2g_{_F}}}
{\int_{P\times F} \tilde{\varphi}^2\ d\mu_{g_{_P}+\tilde{\rho}^2g_{_F}}}|<\epsilon.
\end{eqnarray}
Moreover taking $\tilde{\varphi}$ further closer to $\tilde{\pi}^*\varphi$ we can also achieve
\begin{eqnarray}\label{stopphonecall}
\int_{P\times F}\tilde{\varphi}\ d\mu_{g_{_P}+\tilde{\rho}^2g_{_F}}=0
\end{eqnarray}
by using $$\int_{P\times F}\tilde{\pi}^*\varphi\ d\mu_{g_{_P}+\tilde{\rho}^2g_{_F}}=|G|\int_X \varphi\ d\mu_{g_{_X}^\rho}=0$$ and adding a small constant to $\tilde{\varphi}$ as necessary while not violating the above three inequalities (\ref{body-landing-1}, \ref{body-landing-2}, \ref{body-landing}) and not harming the property of being a generic-fiberwise Morse function.

Now take the fiberwise spherical rearrangement $\tilde{\varphi}_{\bar{*}}$ of $\tilde{\varphi}$ on the Riemannian product $(P\times F, g_{_P}+g_{_F})$, then the resulting $\tilde{\varphi}_{\bar{*}}\in L_1^2(P\times S^m_V)\cap C^0(P\times S^m_V)$ due to Theorem \ref{main-estimates} is ``almost $G$-invariant", since $\tilde{\varphi}$ was ``almost $G$-invariant", where a $G$-action on $P\times S^m_V$ is defined such that it acts only on $P$ keeping $S^m_V$. Indeed
\begin{eqnarray}\label{ChoeChoi}
\sup_{g\in G}||\tilde{\varphi}_{\bar{*}}-(\tilde{\varphi}_{\bar{*}})\cdot g||_{\infty}&=&
\sup_{g\in G}||\tilde{\varphi}_{\bar{*}}-(\tilde{\varphi}\cdot g)_{\bar{*}}||_{\infty}\nonumber\\ &\leq& \sup_{g\in G}||\tilde{\varphi}-\tilde{\varphi}\cdot g||_{\infty}\nonumber\\ &\leq& \sup_{g\in G}(||\tilde{\varphi}-\tilde{\pi}^*\varphi||_{\infty}+||\tilde{\pi}^*\varphi-\tilde{\varphi}\cdot g||_{\infty})\nonumber\\ &<& \epsilon
\end{eqnarray}
where the 1st equality is due to (\ref{Ocean-nuclear}) and the inequality of the 2nd line is due to Theorem \ref{Firstlady}.

Let's define $$\tilde{\varphi}_{\circledast}:P\times S^m_V\rightarrow \Bbb R$$ by $$\tilde{\varphi}_{\circledast}(x):=\frac{\int_{G} (\tilde{\varphi}_{\bar{*}}\cdot g)(x)\ d\mu_{g_{_G}}(g)}{|G|}$$ using the notation of (\ref{BHCP}).
To see that this integral is well-defined, we rewrite $$|G|\tilde{\varphi}_{\circledast}(x)=\int_{G\times \{x\}} \tilde{\varphi}_{\bar{*}}\circ \sigma\ d\mu_{g_{_G}}$$
where $\sigma:G\times (P\times S^m_V)\rightarrow P\times S^m_V$ is given by $\sigma(g,x)=x\cdot g^{-1}$.
Since $\tilde{\varphi}_{\bar{*}}$ is continuous, so is $\tilde{\varphi}_{\bar{*}}\circ \sigma$, and its integration is well-defined.

\begin{Lemma}\label{grape-heaven}
$\tilde{\varphi}_{\circledast}$ belongs to $L_1^2(P\times S^m_V)$ and is $G$-invariant. In particular $\tilde{\varphi}_{\circledast}$ projects down to a function $\check{\varphi}:B\times S^m_V\rightarrow \Bbb R$ which is weakly differentiable.
\end{Lemma}
\begin{proof}
Being in $L_1^2$ is the consequence of Proposition \ref{gamsa1}. To see the $G$-invariance, for any $x\in P\times S^m_V$ and $g'\in G$,
\begin{eqnarray*}
|G|\tilde{\varphi}_{\circledast}(x\cdot g')&=&
\int_{G} (\tilde{\varphi}_{\bar{*}}\cdot g)(x\cdot g')\ d\mu_{g_{_G}}(g)\\ &=&
\int_G\tilde{\varphi}_{\bar{*}}((x\cdot g')\cdot g^{-1})\ d\mu_{g_{_G}}(g)\\ &=&
\int_G(\tilde{\varphi}_{\bar{*}}\cdot (g(g')^{-1}))(x)\ d\mu_{g_{_G}}(g(g')^{-1})\\ &=&|G|\ \tilde{\varphi}_{\circledast}(x)
\end{eqnarray*}
by using the right-invariance of $g_{_G}$.

To show the weak differentiability of $\check{\varphi}$, take local coordinates $U_\alpha\subset B$ and $W\subset S^m_V$ such that $\pi_{_P}^{-1}(U_\alpha)\simeq U_\alpha\times G$. Then in $U_\alpha\times G\times W$, $\tilde{\varphi}_{\circledast}$ is just the obvious extension of
$\check{\varphi}$ defined on $U_\alpha\times \{e\}\times W$. The weak differentiability of $\check{\varphi}$ is immediate from that of $\tilde{\varphi}_{\circledast}$ by applying Proposition \ref{Sewoong}.
\end{proof}

So we have obtained
$$\varphi \Rightarrow \tilde{\pi}^*\varphi \Rightarrow \tilde{\varphi} \Rightarrow  \tilde{\varphi}_{\bar{*}} \Rightarrow \tilde{\varphi}_{\circledast} \Rightarrow \check{\varphi}$$
and $\check{\varphi}$ is the desired function on $B\times S^m_V$. To compare its Rayleigh quotient with that of $\varphi$, it is now the main task to find relations between norms of these functions.
% and in particular $\tilde{\varphi}_{\circledast}$. and not only $\tilde{\varphi}_{\bar{*}}$ but also $\check{\varphi}$.
The followings are devoted to this purpose. Let's denote norms of differential forms on $(P\times S^m,c^2g_{_P}+\tilde{\rho}^2\textsl{g}_V)$ and $(B\times S^m,c^2g_{_B}+\check{\rho}^2\textsl{g}_V)$
by $|\cdot|_{c,\tilde{\rho}}$ and $|\cdot|_{c,\check{\rho}}$ respectively.

First of all, through the process $\tilde{\varphi}_{\circledast} \Rightarrow \check{\varphi}$ the values of Rayleigh quotient do not change basically, since $\tilde{\varphi}_{\circledast}$ is just the lift of $\check{\varphi}$. By applying Proposition \ref{Matt-Damon} to
$$\tilde{\pi}:(P\times S^m,c^2g_{_P}+\tilde{\rho}^2\textsl{g}_V) \rightarrow (B\times S^m,c^2g_{_B}+\check{\rho}^2\textsl{g}_V)$$
$$(p,y)\mapsto (\pi_{_P}(p),y)$$
where we abused the notation $\tilde{\pi}$ following the proposition, we get
\begin{eqnarray}\label{sunmok2}
\int_{P\times S^m} \tilde{\varphi}_{\circledast}^k\ d\mu_{c^2g_{_P}+\tilde{\rho}^2\textsl{g}_V}=c^{\ell}|G|\int_{B\times S^m} \check{\varphi}^k\ d\mu_{c^2g_{_B}+\check{\rho}^2\textsl{g}_V}
\end{eqnarray}
for $k=1,2$, and \footnote{In this case the projection map $\tilde{\pi}$ is a Riemannian submersion, because the $G$-action on $S^m$ is trivial. So  (\ref{sunmok3}) can be proved without recourse to the 2nd equality of Proposition \ref{Matt-Damon}. Indeed by using the fact that $\tilde{\varphi}_{\circledast}$ is constant along any $G$-orbit,
$$|d\tilde{\varphi}_{\circledast}\wedge c^\ell\Omega_G|_{c,\tilde{\rho}}=|d\tilde{\varphi}_{\circledast}|_{c,\tilde{\rho}}=\tilde{\pi}^*|d\check{\varphi}|^2_{c,\check{\rho}}$$
holds a.e., namely at any point where $\tilde{\varphi}_{\circledast}$ is differentiable.}
\begin{eqnarray}\label{sunmok3}
\int_{P\times S^m} |d\tilde{\varphi}_{\circledast}\wedge c^\ell\Omega_G|^2_{c,\tilde{\rho}}\ d\mu_{c^2g_{_P}+\tilde{\rho}^2\textsl{g}_V}
&=& \int_{P\times S^m} |d(\tilde{\pi}^*\check{\varphi})\wedge c^\ell\Omega_G|^2_{c,\tilde{\rho}}\ d\mu_{c^2g_{_P}+\tilde{\rho}^2\textsl{g}_V}\nonumber\\
&=& c^{\ell}|G|\int_{B\times S^m} |d\check{\varphi}|^2_{c,\check{\rho}}\ d\mu_{c^2g_{_B}+\check{\rho}^2\textsl{g}_V}.
\end{eqnarray}

Secondly, the $L^2$-norm of $\tilde{\varphi}_{\circledast}$ should be close to that of $\tilde{\varphi}_{\bar{*}}$. Concretely we shall need the following.
\begin{Lemma}\label{leejung1}
\begin{eqnarray}
\left(\int_{P\times S^m}(\tilde{\varphi}_{\circledast})^2\ d\mu_{c^2g_{_P}+\tilde{\rho}^2\textsl{g}_V}\right)^{\frac{1}{2}} &\geq&
\left(\int_{P\times S^m}(\tilde{\varphi}_{\bar{*}})^2\ d\mu_{c^2g_{_P}+\tilde{\rho}^2\textsl{g}_V}\right)^{\frac{1}{2}}-\tilde{C}\epsilon\\
&>& \frac{\sqrt{c^{n+\ell}|G|}}{2}.
\end{eqnarray}
\end{Lemma}
\begin{proof}
%Using a local trivialization $\pi^{-1}_P(U_\alpha)\simeq U_\alpha\times G$ of $P$, for any $(b,h,y)\in U_\alpha\times G\times S^m_V$
%\begin{eqnarray*}
%\tilde{\varphi}_{\circledast}(b,h,y)&=& \tilde{\varphi}_{\circledast}(b,e,y)\\ &=&\frac{1}{|G|}\int_{G} \tilde{\varphi}_{\bar{*}}(b,g,y)\ d\mu_{g_{_G}}(g)
%\end{eqnarray*}
%where $e$ is the identity element of $G$, and hence
For any $z\in P\times S^m$
\begin{eqnarray*}
|\tilde{\varphi}_{\bar{*}}(z)-\tilde{\varphi}_{\circledast}(z)| &=& \frac{1}{|G|}|\tilde{\varphi}_{\bar{*}}(z)\int_{G} d\mu_{g_{_G}}(g)-\int_{G} (\tilde{\varphi}_{\bar{*}}\cdot g)(z)\ d\mu_{g_{_G}}(g)|\\ &\leq&\frac{1}{|G|}
\int_{G} |\tilde{\varphi}_{\bar{*}}(z)- (\tilde{\varphi}_{\bar{*}}\cdot g)(z)|\ d\mu_{g_{_G}}(g)\\ &\leq& \epsilon
\end{eqnarray*}
by (\ref{ChoeChoi}).
Thus
$$\int_{P\times S^m}(\tilde{\varphi}_{\circledast}-\tilde{\varphi}_{\bar{*}})^2\ d\mu_{c^2g_{_P}+\tilde{\rho}^2\textsl{g}_V}\leq
\tilde{C}^2\epsilon^2,$$
and hence
\begin{eqnarray*}
\left(\int_{P\times S^m}(\tilde{\varphi}_{\circledast})^2\ d\mu_{c^2g_{_P}+\tilde{\rho}^2\textsl{g}_V}\right)^{\frac{1}{2}} &\geq&
\left(\int_{P\times S^m}(\tilde{\varphi}_{\bar{*}})^2\ d\mu_{c^2g_{_P}+\tilde{\rho}^2\textsl{g}_V}\right)^{\frac{1}{2}}-\tilde{C}\epsilon\\&=&
\left(\int_{P\times F}\tilde{\varphi}^2\ d\mu_{c^2g_{_P}+\tilde{\rho}^2g_{_F}}\right)^{\frac{1}{2}}-\tilde{C}\epsilon\\
&=& \left(c^{n+\ell}\int_{P\times F}\tilde{\varphi}^2\ d\mu_{g_{_P}+\tilde{\rho}^2g_{_F}}\right)^{\frac{1}{2}}-\tilde{C}\epsilon\\
&>& \sqrt{c^{n+\ell}}(\left(\int_{P\times F}\tilde{\pi}^*\varphi^2\ d\mu_{g_{_P}+\tilde{\rho}^2g_{_F}}\right)^{\frac{1}{2}}-\epsilon)-\tilde{C}\epsilon\\
&=& \sqrt{c^{n+\ell}}\left(|G|\int_{X}\varphi^2\ d\mu_{g_{_X}^\rho}\right)^{\frac{1}{2}}-(\sqrt{c^{n+\ell}}+\tilde{C})\epsilon\\
&=& \sqrt{c^{n+\ell}|G|}-(\sqrt{c^{n+\ell}}+\tilde{C})\epsilon\\
&>& \frac{\sqrt{c^{n+\ell}|G|}}{2}
\end{eqnarray*}
by our choice of sufficiently small $\epsilon>0$, where the inequality of the 4th line is obtained by the combination of (\ref{body-landing-2}) and the Minkowski inequality.
By a similar computation one can also get that  $(\int_{P\times S^m}(\tilde{\varphi}_{\circledast})^2\ d\mu_{c^2g_{_P}+\tilde{\rho}^2\textsl{g}_V})^{\frac{1}{2}}$ is bounded above by a constant.
\end{proof}

Thirdly, through the symmetrizing processes $\tilde{\varphi} \Rightarrow  \tilde{\varphi}_{\bar{*}} \Rightarrow \tilde{\varphi}_{\circledast}$, the norms of derivatives are basically decreasing, as expected.
%As the last preparation for $\lambda_1$ comparison,
\begin{Lemma}\label{leejung2}
\begin{eqnarray*}
\int_{P\times F} |d\tilde{\varphi}\wedge \Omega_G|^2_{\tilde{\rho}}\ d\mu_{c^2g_{_P}+\tilde{\rho}^2g_{_F}}&\geq&
c^2\int_{P\times S^m} |d\tilde{\varphi}_{\bar{*}}\wedge c^\ell\Omega_G|^2_{c,\tilde{\rho}}\ d\mu_{c^2g_{_P}+\tilde{\rho}^2\textsl{g}_V}\\
&\geq& c^2\int_{P\times S^m}|d\tilde{\varphi}_{\circledast}\wedge c^\ell\Omega_G|^2_{c,\tilde{\rho}}\ d\mu_{c^2g_{_P}+\tilde{\rho}^2\textsl{g}_V}.
\end{eqnarray*}
\end{Lemma}
\begin{proof}
Recall that $|\cdot|_{\tilde{\rho}}$ denotes norm w.r.t. $g_{_P}+\tilde{\rho}^2g_{_F}$.
For the 1st inequality, it's enough to show that
\begin{eqnarray}\label{puskas}
\int_{\{p\}\times F} |d\tilde{\varphi}\wedge \Omega_G|^2_{\tilde{\rho}}\ d\mu_{g_{_F}}\geq
c^2\int_{\{p\}\times S^m} |d\tilde{\varphi}_{\bar{*}}\wedge c^\ell\Omega_G|^2_{c,\tilde{\rho}}\ d\mu_{\textsl{g}_V}
\end{eqnarray}
for any point $p\in N_0(\tilde{\varphi})$.
%If it were not for $\Omega_G$ and $c^\ell\Omega_G$, the inequality would be the consequence of (\ref{JHK1}, \ref{JHK2}) in the proof of Proposition \ref{Yoon-1}. As in there,

We split the exterior derivative $d$ on $P\times F$ and $P\times S^m$ as $d^P+d^F$ and $d^P+d^S$ respectively according to their product structure.
Let $\{v_1,\cdots,v_{n+\ell}\}$ be an orthonormal basis of $T_pP$ w.r.t. metric $g_{_P}$ such that $\Omega_G$ at $p$ is $v_{n+1}^*\wedge\cdots\wedge v_{n+\ell}^*$
where $v_i^*$ means the dual cotangent vector of $v_i$.
By using the inequalities in Theorem \ref{main-estimates}, we get
\begin{eqnarray*}
\int_{\{p\}\times F} |d^P\tilde{\varphi}\wedge \Omega_G|^2_{\tilde{\rho}}\ d\mu_{g_{_F}}&=& \int_{\{p\}\times F} |d^P\tilde{\varphi}\wedge \Omega_G|^2_{g_{_P}} d\mu_{g_{_F}}\\ &=& \sum_{i=1}^{n} \int_{\{p\}\times F} |v_i(\tilde{\varphi})|^2 d\mu_{g_{_F}} \\
&\geq&  \sum_{i=1}^{n} \int_{\{p\}\times S^m_V} |v_i(\tilde{\varphi}_{\bar{*}})|^2 d\mu_{\textsl{g}_V} \\ &=&
\int_{\{p\}\times S^m_V} |d^P\tilde{\varphi}_{\bar{*}}\wedge \Omega_G|^2_{g_{_P}} d\mu_{\textsl{g}_V}\\
&=& c^2\int_{\{p\}\times S^m} |d^P\tilde{\varphi}_{\bar{*}}\wedge c^\ell\Omega_G|^2_{c^2g_{_P}} d\mu_{\textsl{g}_V}\\
&=& c^2\int_{\{p\}\times S^m} |d^P\tilde{\varphi}_{\bar{*}}\wedge c^\ell\Omega_G|^2_{c,\tilde{\rho}}\ d\mu_{\textsl{g}_V},
\end{eqnarray*}
%where $v_i(\cdot)$ denotes the directional derivative w.r.t. $v_i$
and
\begin{eqnarray*}
\int_{\{p\}\times F} |d^F\tilde{\varphi}\wedge \Omega_G|^2_{\tilde{\rho}}\ d\mu_{g_{_F}} &=&
%(\tilde{\rho}(p))^{-2}\int_{\{p\}\times F} |d^F\tilde{\varphi}\wedge \Omega_G|^2_{g_{_F}} d\mu_{g_{_F}} \\ &=&
(\tilde{\rho}(p))^{-2}\int_{\{p\}\times F} |d^F\tilde{\varphi}|^2_{g_{_F}} d\mu_{g_{_F}}\\ &\geq& (\tilde{\rho}(p))^{-2}c^2\int_{\{p\}\times S^m_V} |d^S\tilde{\varphi}_{\bar{*}}|^2_{\textsl{g}_V} d\mu_{\textsl{g}_V}\\ &=& c^2\int_{\{p\}\times S^m} |d^S\tilde{\varphi}_{\bar{*}}\wedge c^\ell\Omega_G|^2_{c,\tilde{\rho}}\ d\mu_{\textsl{g}_V}.
\end{eqnarray*}
By combining the above two inequalities, (\ref{puskas}) follows.

The 2nd inequality is obtained by applying Proposition \ref{gamsa} with $\Omega$ equal to $c^\ell\Omega_G$.
\end{proof}

Now we are ready to compare $\lambda_1(X,g_{_X}^\rho)$ and $\lambda_1(B\times S^m,c^2g_{_B}+\check{\rho}^2\textsl{g}_V)$ by combining all we have obtained so far.
By Proposition \ref{Matt-Damon} and (\ref{body-landing})
\begin{eqnarray*}
\lambda_1(X,g_{_X}^\rho) &=& R_{g_{_X}^\rho} (\varphi)\\ &=& \frac{\int_X | d\varphi|^2_\rho\ d\mu_{g_{_X}^\rho}}{\int_X\varphi^2\ d\mu_{g_{_X}^\rho}}\\ &=&
\frac{\int_{P\times F} |d(\tilde{\pi}^*\varphi)\wedge \Omega_G|^2_{\tilde{\rho}}\ d\mu_{g_{_P}+\tilde{\rho}^2g_{_F}}}
{\int_{P\times F} (\tilde{\pi}^*\varphi)^2\ d\mu_{g_{_P}+\tilde{\rho}^2g_{_F}}}\\
&\geq& \frac{\int_{P\times F} |d\tilde{\varphi}\wedge \Omega_G|^2_{\tilde{\rho}}\ d\mu_{g_{_P}+\tilde{\rho}^2g_{_F}}}
{\int_{P\times F} \tilde{\varphi}^2\ d\mu_{g_{_P}+\tilde{\rho}^2g_{_F}}}-\epsilon\\
&=& \frac{\int_{P\times F} |d\tilde{\varphi}\wedge \Omega_G|^2_{\tilde{\rho}}\ d\mu_{c^2g_{_P}+\tilde{\rho}^2g_{_F}}}
{\int_{P\times F} \tilde{\varphi}^2\ d\mu_{c^2g_{_P}+\tilde{\rho}^2g_{_F}}}-\epsilon.
\end{eqnarray*}
Upon taking the fiberwise spherical rearrangement $\tilde{\varphi}_{\bar{*}}$ of $\tilde{\varphi}$ and applying Lemma \ref{leejung2}, Lemma \ref{leejung1} and equalities (\ref{libraryinNY}), (\ref{sunmok2}), (\ref{sunmok3}), the above last term is
\begin{eqnarray*}
\ &\geq& c^2
\frac{\int_{P\times S^m} |d\tilde{\varphi}_{\bar{*}}\wedge c^\ell\Omega_G|^2_{c,\tilde{\rho}}\ d\mu_{c^2g_{_P}+\tilde{\rho}^2\textsl{g}_V}}
{\int_{P\times S^m} (\tilde{\varphi}_{\bar{*}})^2\ d\mu_{c^2g_{_P}+\tilde{\rho}^2\textsl{g}_V}}-\epsilon\\ &\geq& c^2
\frac{\int_{P\times S^m} |d\tilde{\varphi}_{\circledast}\wedge c^\ell\Omega_G|^2_{c,\tilde{\rho}}\ d\mu_{c^2g_{_P}+\tilde{\rho}^2\textsl{g}_V}}
{\left(\left(\int_{P\times S^m}(\tilde{\varphi}_{\circledast})^2\ d\mu_{c^2g_{_P}+\tilde{\rho}^2\textsl{g}_V}\right)^{\frac{1}{2}}+
\tilde{C}\epsilon\right)^2}-\epsilon
\\ &\geq& c^2
\frac{\int_{P\times S^m} |d\tilde{\varphi}_{\circledast}\wedge c^\ell\Omega_G|^2_{c,\tilde{\rho}}\ d\mu_{c^2g_{_P}+\tilde{\rho}^2\textsl{g}_V}}
{\int_{P\times S^m}(\tilde{\varphi}_{\circledast})^2\ d\mu_{c^2g_{_P}+\tilde{\rho}^2\textsl{g}_V}}
\left(1-\frac{\tilde{C}\epsilon}{\left(\int_{P\times S^m}(\tilde{\varphi}_{\circledast})^2\  d\mu_{c^2g_{_P}+\tilde{\rho}^2\textsl{g}_V}\right)^{\frac{1}{2}}}\right)^2
-\epsilon \\ &\geq& c^2
\frac{c^\ell |G|\int_{B\times S^m} |d\check{\varphi}|^2_{c,\check{\rho}}\ d\mu_{c^2g_{_B}+\check{\rho}^2\textsl{g}_V}}{c^\ell |G|\int_{B\times S^m}\check{\varphi}^2\ d\mu_{c^2g_{_B}+\check{\rho}^2\textsl{g}_V}}
\left(1-\frac{\tilde{C}\epsilon}{\frac{\sqrt{c^{n+\ell}|G|}}{2}}\right)^2-\epsilon.
\end{eqnarray*}
Here the 3rd line is obtained by applying the inequality $$\frac{1}{a+b}\geq \frac{1}{a}(1-\frac{b}{a})$$ for any $a>b\geq 0$ combined with $$(\int_{P\times S^m}(\tilde{\varphi}_{\circledast})^2\  d\mu_{c^2g_{_P}+\tilde{\rho}^2\textsl{g}_V})^{\frac{1}{2}}> \frac{\sqrt{c^{n+\ell}|G|}}{2} > \tilde{C}\epsilon$$ which was also used in the 4th line. The last term can be simplified as
$$c^2R_{c^2g_{_B}+\check{\rho}^2\textsl{g}_V}(\check{\varphi})\ (1-C'\epsilon)^2-\epsilon$$
where $C'$ is a constant independent of $\epsilon>0$.

It still remains to show that $\check{\varphi}$ is $L^2$-orthogonal to 1, while its $L^2$-norm is away from 0.
Indeed
\begin{eqnarray*}
\int_{B\times S^m} \check{\varphi}^2\ d\mu_{c^2g_{_B}+\check{\rho}^2\textsl{g}_V}&=&
\frac{1}{c^{\ell}|G|}\int_{P\times S^m} \tilde{\varphi}_{\circledast}^2\ d\mu_{c^2g_{_P}+\tilde{\rho}^2\textsl{g}_V}\\
&>&  \frac{1}{c^{\ell}|G|}(\frac{\sqrt{c^{n+\ell}|G|}}{2})^2\\ &=& \frac{c^n}{4},
%(\sqrt{c^{n+\ell}|G|}-\epsilon-\epsilon\ \sqrt{\mu(P\times S^m)}\ )^2.
\end{eqnarray*}
and
\begin{eqnarray*}
\int_{B\times S^m} \check{\varphi}\ d\mu_{c^2g_{_B}+\check{\rho}^2\textsl{g}_V}&=&
\frac{1}{c^{\ell}|G|} \int_{P\times S^m} \tilde{\varphi}_{\circledast}\ d\mu_{c^2g_{_P}+\tilde{\rho}^2\textsl{g}_V}\\
&=& \frac{1}{c^{\ell}|G|^2} \int_{P\times S^m}  \int_{G} \tilde{\varphi}_{\bar{*}}\cdot g\ d\mu_{g_{_G}}(g) d\mu_{c^2g_{_P}+\tilde{\rho}^2\textsl{g}_V}\\
&=& \frac{1}{c^{\ell}|G|^2} \int_{G}\int_{P\times S^m}\tilde{\varphi}_{\bar{*}}\cdot g\  d\mu_{c^2g_{_P}+\tilde{\rho}^2\textsl{g}_V}d\mu_{g_{_G}}(g)\\
&=& \frac{1}{c^{\ell}|G|^2} \int_{G}\int_{P\times S^m}  \tilde{\varphi}_{\bar{*}}\ d\mu_{c^2g_{_P}+\tilde{\rho}^2\textsl{g}_V}d\mu_{g_{_G}}(g)\\
&=& \frac{1}{c^{\ell}|G|}  \int_P\int_{S^m}  \tilde{\varphi}_{\bar{*}}\ d\mu_{\tilde{\rho}^2\textsl{g}_V}d\mu_{c^2g_{_P}}\\
&=& \frac{1}{c^{\ell}|G|}  \int_P\int_F  \tilde{\varphi}\ d\mu_{\tilde{\rho}^2g_{_F}}d\mu_{c^2g_{_P}}\\
&=& \frac{c^{n}}{|G|} \int_{P\times F}  \tilde{\varphi}\ d\mu_{g_{_P}+\tilde{\rho}^2g_{_F}}\\
&=& 0
\end{eqnarray*}
where the 4th equality is due to the fact that the $G$-action on $(P\times S^m,c^2g_{_P}+\tilde{\rho}^2\textsl{g}_V)$ is isometric and the 6th equality is due to (\ref{libraryinNY}).
%Thus we can apply Lemma \ref{Gutentag} to get $$\lambda_1(X,g_{_X}^\rho)\geq c^2R_{c^2g_{_B}+\check{\rho}^2\textsl{g}_V}(\check{\varphi}^\perp)(1-C''\epsilon^2)\ (1-C'\epsilon)^2-\epsilon$$ for a constant $C''>0$.
Finally we can conclude that
\begin{eqnarray*}
\lambda_1(X,g_{_X}^\rho)&\geq& c^2\lambda_1(B\times S^m,c^2g_{_B}+\check{\rho}^2\textsl{g}_V)\ (1-C'\epsilon)^2-\epsilon \\
&=&  \lambda_1(B\times S^m,g_{_B}+\check{\rho}^2g_{_{\Bbb S}})\ (1-C'\epsilon)^2-\epsilon,
\end{eqnarray*}
which gives the desired inequality by taking $\epsilon>0$ arbitrarily small.

\subsection{Proof of Theorem \ref{Main-Corollary}}

Since $\rho$ is constant, $\pi: (X,g_{_X}^\rho)\rightarrow (B,g_{_B})$ is still a Riemannian submersion with totally geodesic fibers. So each fiber is a minimal submanifold and we can apply Watson's theorem \cite{Watson}.

Let $\{\psi_i|i\in \Bbb N\cup \{0\}\}$ be a complete orthonormal basis of $L^2(B)$ consisting of eigenfunctions of $\Delta_B$ with $\psi_i$ having eigenvalue $\lambda_i(B,g_{_B})$ where $\psi_0$ is a constant.
For each $i$, $\pi^*\psi_i$ is an eigenfunction of $(X,g_{_X}^\rho)$ with the same eigenvalue as $\psi_i$ and for $j\ne i$
$$\int_X \pi^*\psi_i\ \pi^*\psi_j\ d\mu_{g_{_X}^\rho}=\int_B \psi_i\ \psi_j\ d\mu_{g_{_B}}\ \rho^m\int_F d\mu_{g_{_F}}=0$$ so that
\begin{eqnarray}\label{Thanks-PNU-0}
\lambda_i(B,g_{_B}) \geq \lambda_i(X,g_{_X}^\rho).
\end{eqnarray}

Likewise let $\{\varphi_i|i\in \Bbb N\cup \{0\}\}$ be such a basis of $L^2(X)$ with $\varphi_i$ having eigenvalue $\lambda_i(X,g_{_X}^\rho)$. We wish to show that for $i\leq I$
$$\lambda_i(B,g_{_B})= \lambda_i(X,g_{_X}^\rho)$$
and $\varphi_i$ can be chosen to be a constant multiple of $\pi^*\psi_i$. We prove this by induction on $i$.

First the case of $i=0$ is obvious, since the 0-th eigenvalue is 0 and the 0-th eigenfunction is a constant function on both $B$ and $X$.
%By Watson's theorem and the conclusion of Theorem \ref{Main-theorem}
%\begin{eqnarray*}
%\lambda_1(B,g_{_B}) &\geq& \lambda_1(X,g_{_X}^\rho) \\ &\geq& \lambda_1(B\times S^m, g_{_B}+\rho^2g_{_{\Bbb S}})\\  &=& \min(\lambda_1(B,g_{_B}),\frac{m}{\rho^2})\\ &=& \lambda_1(B,g_{_B}),
%\end{eqnarray*}
%thereby producing the desired equality. Since $\pi^*\psi_1$ is already an eigenfunction on $X$ with eigenvalue equal to $\lambda_1(B,g_{_B})$, it must be a constant multiple of $\varphi_1$.
Let's suppose that it is proved up to $k\leq I-1$. Then $\pi^*\psi_0,\cdots, \pi^*\psi_k$ are some constant multiples of $\varphi_0,\cdots,\varphi_k$ respectively, so let's let $\pi^*\psi_i=a_i\varphi_i$ for each $i$.

Consider $\varphi_{k+1}$ which is not only of unit $L^2$-norm but also orthogonal to $\varphi_0,\cdots, \varphi_k$. We go through the process in the proof of Theorem \ref{Main-theorem} with $\varphi_{k+1}$ instead of $\varphi$. Then in the same way we get
\begin{eqnarray}\label{Thanks-PNU-1}
\lambda_{k+1}(X,g_{_X}^\rho)=R_{g_{_X}^\rho} (\varphi_{k+1})\geq c^2R_{c^2g_{_B}+\check{\rho}^2\textsl{g}_V}(\check{\varphi})\ (1-C'\epsilon)^2-\epsilon.
\end{eqnarray}

However this is not enough for finishing the proof. In case of higher eigenvalues one more condition needs to be checked to derive the conclusion. It is true that $\check{\varphi}$ is orthogonal to $\psi_0$, but it may not be orthogonal to $\psi_1,\cdots, \psi_k$, to which $\check{\varphi}-\sum_{i=1}^k\langle \psi_i, \check{\varphi}\rangle\psi_i$ is orthogonal. Nevertheless $\check{\varphi}$ is almost orthogonal to them in the sense that $\sum_{i=1}^k\langle \psi_i, \check{\varphi}\rangle\psi_i$ has small $L^2$-norm, and this is enough to estimate $\lambda_{k+1}(B\times S^m,c^2g_{_B}+\check{\rho}^2\textsl{g}_V)$ by using
(\ref{Thanks-PNU-1}) and $\check{\varphi}-\sum_{i=1}^k\langle \psi_i, \check{\varphi}\rangle\psi_i$ which is orthogonal to not only $\psi_0$ but also $\psi_1,\cdots, \psi_k$.

Indeed for $i=1,\cdots, k$
\begin{eqnarray*}
\int_{B\times S^m} \psi_i\ \check{\varphi}\ d\mu_{c^2g_{_B}+\check{\rho}^2\textsl{g}_V}&=&
\frac{1}{c^{\ell}|G|} \int_{P\times S^m} \pi_{_P}^*\psi_i\ \tilde{\varphi}_{\circledast}\  d\mu_{c^2g_{_P}+\tilde{\rho}^2\textsl{g}_V}\\
&=& \frac{1}{c^{\ell}|G|^2} \int_{P\times S^m} \pi_{_P}^*\psi_i \int_{G} \tilde{\varphi}_{\bar{*}}\cdot g\ d\mu_{g_{_G}}(g) d\mu_{c^2g_{_P}+\tilde{\rho}^2\textsl{g}_V}\\
&=& \frac{1}{c^{\ell}|G|^2} \int_{G}\int_{P\times S^m}(\pi_{_P}^*\psi_i)\cdot g\ \tilde{\varphi}_{\bar{*}}\cdot g\  d\mu_{c^2g_{_P}+\tilde{\rho}^2\textsl{g}_V}d\mu_{g_{_G}}(g)\\
&=& \frac{1}{c^{\ell}|G|^2}\int_{G}\int_{P\times S^m} \pi_{_P}^*\psi_i\ \tilde{\varphi}_{\bar{*}}\ d\mu_{c^2g_{_P}+\tilde{\rho}^2\textsl{g}_V}d\mu_{g_{_G}}(g)\\
&=& \frac{1}{c^{\ell}|G|} \int_{P}\pi_{_P}^*\psi_i\int_{S^m} \tilde{\varphi}_{\bar{*}}\ d\mu_{\tilde{\rho}^2\textsl{g}_V}d\mu_{c^2g_{_P}}\\
&=& \frac{1}{c^{\ell}|G|}\int_{P}\pi_{_P}^*\psi_i\int_F  \tilde{\varphi}\ d\mu_{\tilde{\rho}^2g_{_F}}d\mu_{c^2g_{_P}}\\
&=& \frac{c^{n}}{|G|}\int_{P\times F} \tilde{\pi}^*(\pi^*\psi_i)\ \tilde{\varphi}\ d\mu_{g_{_P}+\tilde{\rho}^2g_{_F}}\\
&=& \frac{c^{n}a_i}{|G|}\int_{P\times F} \tilde{\pi}^*\varphi_i\ \tilde{\varphi}\ d\mu_{g_{_P}+\tilde{\rho}^2g_{_F}}
\end{eqnarray*}
where the 3rd equality, the 4th equality, and the 6th equality are respectively due to the $G$-invariance of $\pi_{_P}^*\psi_i$, the isometric property of the $G$-action on $(P\times S^m, c^2g_{_P}+\tilde{\rho}^2\textsl{g}_V)$, and  (\ref{libraryinNY}). Since $\tilde{\varphi}$ was chosen so that $||\tilde{\varphi}-\tilde{\pi}^*\varphi_{k+1}||_{L^2}<\epsilon$ by (\ref{body-landing-2}),
\begin{eqnarray*}
|\langle \tilde{\varphi},\tilde{\pi}^*\varphi_{i}\rangle_{L^2}| &=&
|\langle \tilde{\pi}^*\varphi_{k+1}, \tilde{\pi}^*\varphi_i \rangle_{L^2}+\langle \tilde{\varphi}-\tilde{\pi}^*\varphi_{k+1}, \tilde{\pi}^*\varphi_i \rangle_{L^2}| \\
&\leq& |G|\ |\langle \varphi_{k+1}, \varphi_i \rangle_{L^2(X)}|+||\tilde{\varphi}-\tilde{\pi}^*\varphi_{k+1}||_{L^2} ||\tilde{\pi}^*\varphi_i ||_{L^2}\\
&=& 0+||\tilde{\varphi}-\tilde{\pi}^*\varphi_{k+1}||_{L^2}\sqrt{|G|}\ ||\varphi_i ||_{L^2(X)}\\
&<& \sqrt{|G|}\epsilon
%\\ &\leq& \frac{\epsilon}{2}\sqrt{\textrm{Vol}(X)}|G|
\end{eqnarray*}
where the subscript $L^2$ means $L^2(P\times F,g_{_P}+\tilde{\rho}^2g_{_F})$. Thus
\begin{eqnarray}\label{toohot}
|\langle \psi_i, \check{\varphi}\rangle_{L^2(B\times S^m,c^2g_{_B}+\check{\rho}^2\textsl{g}_V)}\|< \frac{c^{n}a_i}{\sqrt{|G|}}\epsilon
\end{eqnarray}
for $i=1, \cdots, k$.

In the previous section, it was proved that the $L^2$-norm of $\check{\varphi}$ is greater than $\frac{\sqrt{c^n}}{2}$,
and as remarked at the end of the proof of Lemma \ref {leejung1}, the $L^2$-norms of $\tilde{\varphi}_{\circledast}$ and hence $\check{\varphi}$ are also bounded above by a constant. Then the $L_1^2$-norm of $\check{\varphi}$ is also bounded above by a constant, say $c'$, by (\ref{Thanks-PNU-1}). Thus we can apply Proposition \ref{yam-0} to deduce that there exists positive constants $\varepsilon$ and $\check{C}$ (depending only on $c$ and $c'$) such that
\begin{eqnarray}\label{urine}
||\sum_{i=1}^k\langle \psi_i, \check{\varphi}\rangle \psi_i||_{L_1^2}<\varepsilon
\end{eqnarray}
implies
\begin{eqnarray*}
|R_{c^2g_{_B}+\check{\rho}^2\textsl{g}_V}(\check{\varphi})-R_{c^2g_{_B}+\check{\rho}^2\textsl{g}_V}(\check{\varphi}-\sum_{i=1}^k\langle \psi_i, \check{\varphi}\rangle \psi_i)|&\leq&\check{C}||\sum_{i=1}^k\langle \psi_i, \check{\varphi}\rangle\psi_i||_{L_1^2},
\end{eqnarray*}
which is again less than $\hat{C}\epsilon$ for a constant $\hat{C}>0$ by using (\ref{toohot}). So we assume that $\epsilon$ is always taken small enough to satisfy (\ref{urine}).

Combining this with (\ref{Thanks-PNU-1}) gives
\begin{eqnarray*}
R_{g_{_X}^\rho} (\varphi_{k+1})  &\geq&
c^2 (R_{c^2g_{_B}+\check{\rho}^2\textsl{g}_V}(\check{\varphi}-\sum_{i=1}^k\langle \psi_i, \check{\varphi}\rangle\psi_i)-\hat{C}\epsilon)\ (1-C'\epsilon)^2-\epsilon\\
&\geq& c^2\lambda_{k+1}(B\times S^m,c^2g_{_B}+\check{\rho}^2\textsl{g}_V)\ (1-C'\epsilon)^2-\bar{C}\epsilon \\
&=&  \lambda_{k+1}(B\times S^m,g_{_B}+\check{\rho}^2g_{_{\Bbb S}})\ (1-C'\epsilon)^2-\bar{C}\epsilon,
\end{eqnarray*}
for a constant $\bar{C}>0$, and hence
\begin{eqnarray*}
\lambda_{k+1}(X,g_{_X}^\rho)&\geq& \lambda_{k+1}(B\times S^m,g_{_B}+\rho^2g_{_{\Bbb S}})\\   &=& \lambda_{k+1}(B,g_{_B})
\end{eqnarray*}
where the last equality is obtained by the condition that $$\lambda_{k+1}(B,g_{_B})\leq \lambda_I(B,g_{_B})\leq \frac{m}{\rho^2}=\lambda_{1}(S^m,\rho^2g_{_{\Bbb S}}).$$
Since we also have $\lambda_{k+1}(B,g_{_B}) \geq \lambda_{k+1}(X,g_{_X}^\rho)$
from (\ref{Thanks-PNU-0}), we finally have $$\lambda_{k+1}(B,g_{_B})=\lambda_{k+1}(X,g_{_X}^\rho).$$
Using that $\pi^*\psi_{k+1}$ is already an eigenfunction on $X$ with eigenvalue equal to $\lambda_{k+1}(B,g_{_B})$ and it is orthogonal to $\varphi_{0},\cdots, \varphi_{k}$, we can choose $$\frac{\pi^*\psi_{k+1}}{||\pi^*\psi_{k+1}||_{L^2(X)}}$$ as $\varphi_{k+1}$. This finishes the induction proof.

\begin{Remark}\label{my_sister}
The following fact would be more useful practically.
The same conclusion is obtained when $Ric_{g_{_F}} \geq (m-1)a^2$ and $\lambda_I(B,g_{_B})\leq \frac{m}{\rho^2}a^2$ by considering the homothety by $1/a^2$.
\end{Remark}

\section{Examples of exact computation of $\lambda_i$}\label{Main-Corollary-Example}

Now we compute $\lambda_i$ in some typical fiber bundles by applying the above obtained result. The important thing is the construction of manifolds to which our theorem is applicable, so we shall state only about $\lambda_1$ which is of more interest. The most well-known Riemannian manifolds of positive Ricci curvature would be $\Bbb S^m$ for $m\geq 2$ and $\Bbb CP^n$ for $n\geq 1$ with the Fubini-Study metric.  Recall that $\Bbb S^m=(S^m,g_{_{\Bbb S}})$ denotes the round $m$-sphere of constant curvature 1. Their isometry groups are $O(m+1)$ and $PU(n+1)\rtimes\Bbb Z_2$ respectively, so we look for bundles with these groups as a structure group. First let's consider sphere bundles.

\subsection{Sphere bundles}
For the topological classification of sphere bundles, recall that the orientation-preserving diffeomorphism group $\textrm{Diff}_+(S^m)$ of $S^m$ for $m=1,2,3$ is homotopy-equivalent to $SO(m+1)$ and hence any orientable $S^m$-bundle $X$ over any smooth manifold $B$ for those $m$ arises as an associated bundle of a principal $SO(m+1)$-bundle over $B$.
%So by taking any $SO(m+1)$-connection,  $X$ can be endowed with a Riemannian submersion onto $B$ with totally geodesic fibers isometric to f $X$

Since $\pi_1(SO(3))=\Bbb Z_2$, there are only two kinds of $S^2$-bundles over $S^2$, namely $S^2\times S^2$ and the nontrivial bundle $S^2\tilde{\times}S^2$.
%, and they all arise as the associated bundles of principal $SO(3)$-bundles over $S^2$.
It's well-known that $S^2\tilde{\times}S^2$ can be realized as a holomorphic bundle, namely $\Bbb CP^2\# \overline{\Bbb CP}^2$.
In the same way there are only two $S^3$-bundles over $S^2$,
%and only two $S^3$-bundles over $S^3$,
since $\pi_1(SO(4))=\Bbb Z_2$.
%and $\pi_2(SO(4))\simeq\pi_2(Spin(4))=\pi_2(SU(2)\oplus SU(2))\simeq\pi_2(SO(4))\simeq\Bbb Z_2$ respectively.
For higher $m$, we only consider $S^m$-bundles with structure group $SO(m+1)$ so that they come from principal $SO(m+1)$-bundles.

By taking any $SO(m+1)$-connection on these nontrivial principal bundles, all these nontrivial $S^m$-bundles $S^n\tilde{\times}S^m$ over $S^n$ admit Riemannian submersions onto $\Bbb S^n$ with totally geodesic fibers isometric to $(S^m,\rho^2g_{_{\Bbb S}})$ where $\rho$ is a constant in $(0,\sqrt{\frac{m}{n}}]$, and for those metrics
\begin{eqnarray*}
\lambda_1(S^n\tilde{\times}S^m)&=& \lambda_1(\Bbb S^n)\\ &=&  n
\end{eqnarray*}
by Theorem \ref{Main-Corollary}.

The $S^3$-bundles over $S^4$ are worthy of special attention. Their structure groups are reduced to $SO(4)$, and principal $SO(4)$-bundles over $S^4$ are classified by
%Here everything is in the smooth category. by the fact that $\textrm{Diff}_+(S^3)$ has the homotopy type of $SO(4)$.
\begin{eqnarray*}
\pi_3(SO(4))&=& \pi_3(Spin(4))\\ &=& \pi_3(SU(2)\times SU(2))\\ &=& \pi_3(S^3\times S^3)\\ &=& \Bbb Z\oplus \Bbb Z.
\end{eqnarray*}
However, the diffeomorphic classification of the associated $S^3$-bundles is not so simple. One of them is $S^7$ and some of them are well-known as exotic 7-spheres. While certain families are mutually diffeomorphic, some others can not be homotopically equivalent to each other.  The reader may refer to \cite{crowley}. Nevertheless it is interesting that they all have the same first eigenvalue which is equal to $\lambda_1(\Bbb S^4)= 4$, if the fiber-resizing constant $\rho$ is in $(0,\frac{\sqrt{3}}{2}]$.
%Let's denote each associated $S^3$-bundle by $X_{m,n}$ for $(m,n)\in \Bbb Z\oplus \Bbb Z=\pi_3(SO(4))$. By taking any $SO(4)$-connection, $X_{m,n}$ is endowed with a Riemannian submersion metric onto $(S^4, \frac{1}{4}g_{_{\Bbb S}})$ with totally geodesic fibers isometric to $(S^3,\rho^2g_{_{\Bbb S}})$ where $\rho$ is a constant in $(0,\frac{\sqrt{3}}{4}]$, and for those metrics Corollary \ref{Main-Corollary} gives

A general way of obtaining a Riemannian submersion with totally geodesic fibers isometric to $\Bbb S^{r-1}$ would be the unit sphere bundle $S(E)$ of a real vector bundle $E$ of rank $r\geq 3$ with an inner product. Let $B$ be its base Riemannian manifold and $P_O$ be the principal $O(r)$-bundle consisting of orthonormal frames in $E$. Then $S(E)$ is an $S^{r-1}$-bundle over $B$ associated to $P_O$. By taking any $O(r)$-connection on $P_O$, $S(E)$ can be endowed with a Riemannian submersion over $B$ with totally geodesic fibers isometric to $(S^{r-1},\rho^2g_{_{\Bbb S}})$ where $\rho$ is a constant in $(0,\sqrt{\frac{r-1}{\lambda_1(B)}}]$. For any of these metrics,  we have
$$\lambda_1(S(E))=\lambda_1(B)$$ by Theorem \ref{Main-Corollary}.
In particular one may take $E$ to be the tangent bundle $TB$ of $B$ and the Levi-Civita connection of $B$.

\subsection{Homogeneous CROSS}

As is well-known,
%the spectrum of a homogeneous manifold can be obtained (in principle) algebraically from the corresponding inner product on its Lie algebra and the representation of its
many homogeneous manifolds admit Riemannian submersions with totally geodesic fibers. Indeed for compact Lie groups $K\subset H\subset G$, the obvious projection map $\pi: G/K\rightarrow G/H$ with fiber $H/K$ has such a structure, where all three quotient manifolds are endowed with homogeneous metrics defined as follows \cite{Beber, Besse}.

To define a homogeneous metric on a quotient manifold $G/H$, it's enough to choose an inner product on the tangent space $T_x(G/H)$ at one point $x:=eH \in G/H$, which is invariant under the isotropy representation of $H$ in $T_x(G/H)$. The correspondence
$$V\in T_eG\ \  \mapsto \ \  \frac{d}{dt}|_{t=0}(exp(tV)\cdot x) \in T_x(G/H)$$
gives the identification of an $Ad(H)$-invariant complement $\mathfrak{m}$ to $T_eH$ in $T_eG$ with the adjoint representation of $H$ and $T_x(G/H)$ with the isotropy representation of $H$, and hence there is a one-to-one correspondence between $G$-invariant metrics on $G/H$ and $Ad(H)$-invariant inner product on $\mathfrak{m}$.
%From this a $G$-invariant metric on $G/H$ is induced by the left transitive $G$-action.
Likewise a $H$-invariant homogeneous metric on $H/K$ can be defined from a $Ad(K)$-invariant inner product on an $Ad(K)$-invariant complement $\mathfrak{p}$ to $T_eK$ in $T_eH$, and then a $G$-invariant metric on $G/K$ is induced by the orthogonal direct sum of two inner products on $\mathfrak{p}\oplus \mathfrak{m}$. By the construction $\pi$ automatically satisfies the property of a Riemannian submersion at $eK\in G/K$ and hence everywhere by the $G$-equivariance of $\pi$ and hence $\pi_*$.

We shall consider only homogeneous metrics on a compact rank one symmetric space (CROSS) such as metrics of the generalized Berger spheres and the squashed $\Bbb CP^n$, and apply Theorem \ref{Main-Corollary} to compute their $\lambda_1$ when the fiber has sufficiently positive Ricci curvature. Not only all homogenous metrics on CROSSes but also  their $\lambda_1$ are now completely known \cite{BLP}, so we can check if our computations match well with them. For more examples one can apply the above construction to other homogeneous manifolds in \cite{Besse, Tani, urakawa}.

When $G=Sp(q+1), H=Sp(q)\times Sp(1), K=Sp(q)$ for $q\geq 1$, one can get $$\pi: \Bbb S^{4q+3}\rightarrow (\Bbb HP^q, g_{_{QK}})$$ with fiber isometric to $\Bbb S^3$ where $g_{_{QK}}$ is the quaternionic K\"ahler metric with sectional curvature ranging in $[1,4]$ and $(\Bbb HP^1, g_{_{QK}})$ is isometric to $(S^4,\frac{1}{4}g_{_{\Bbb S}})$.
%(When $q=1$, $(\Bbb HP^q, g_{_{QK}})$ is round $S^4$ with curvature 4.)
Shrinking the fibers constantly by $\rho\in (0,\frac{\sqrt{6}}{4\sqrt{q+1}}]$, $$\lambda_1(\Bbb HP^q, g_{_{QK}})=8(q+1)\leq \frac{3}{\rho^2},$$
and hence this squashed $S^{4q+3}$ has
\begin{eqnarray*}
\lambda_1 = 8(q+1)
\end{eqnarray*}
by Theorem \ref{Main-Corollary}. Our computation coincides with that of S. Tanno \cite{Tanno} whose method computed $\lambda_1$ for all values of $\rho>0$ as follows :
$$\lambda_1=\left\{
  \begin{array}{ll}
    8(q+1) & \hbox{if $\rho\leq\frac{\sqrt{3}}{2\sqrt{q+2}}$} \\
    4q+\frac{3}{\rho^2} & \hbox{otherwise.}
  \end{array}
\right.$$

When $G=Sp(q+1), H=Sp(q)\times Sp(1), K=Sp(q)\times U(1)$ for $q\geq 1$, one can obtain $$\pi: (\Bbb CP^{2q+1},g_{_{FS}}) \rightarrow (\Bbb HP^q, g_{_{QK}})$$ where fibers are round $S^2$ with curvature 4 and $g_{_{FS}}$ denotes the Fubini-Study metric with curvature ranging between 1 and 4.
Shrinking the fibers constantly by $\rho\in (0,\frac{1}{\sqrt{q+1}}]$, $$\lambda_1(\Bbb HP^q, g_{_{QK}})=8(q+1)\leq \frac{2}{\rho^2}\cdot 4,$$
and hence by Remark \ref{my_sister} (of Theorem \ref{Main-Corollary}) this squashed $ \Bbb CP^{2q+1}$ has
\begin{eqnarray*}
\lambda_1 = 8(q+1)
\end{eqnarray*}
which is equal to $\lambda_1(\Bbb CP^{2q+1},g_{_{FS}})$. In this case, our computation coincides with that of Bettiol, Lauret, and Piccione \cite{BLP} whose method computed $\lambda_1$ for all values of $\rho>0$ as follows :
$$\lambda_1=\left\{
  \begin{array}{ll}
    8(q+1) & \hbox{if $\rho\leq 1$} \\
    8(q+\frac{1}{\rho^2}) & \hbox{otherwise.}
  \end{array}
\right.$$
%When $q=1$, we have $$\pi: (\Bbb CP^{3},g_{_{FS}}) \rightarrow (S^4,  \frac{1}{4}g_{_{\Bbb S}}),$$ and the same result still holds.

%It is left as an interesting question to find out how $\lambda_1$ changes for other values of $\rho >0$. considering it as $SU(8)/SU(7)$, it has a Riemannian submersion with totally geodesic fibers $S^1$ of length $2\pi$ onto $\Bbb CP^7\simeq SU(8)/S(U(1)U(7))$ with the Fubnini-Study metric.

%Secondly, considering $\Bbb S^{15}$ as $Sp(4)/Sp(3)$, it has a Riemannian submersion with totally geodesic fibers $\Bbb S^3$ onto $\Bbb HP^3\simeq Sp(4)/Sp(1)Sp(3)$ with the quaternion-K\"ahler metric $g_{_{QK}}$. Shrinking the fibers constantly by $\rho\in (0,\frac{\sqrt{6}}{8}]$, $$\lambda_1(\Bbb HP^3, g_{_{QK}})\leq \frac{3}{\rho^2},$$ and hence this squashed $S^{15}$ has
%\begin{eqnarray*}
%\lambda_1 =\lambda_1(\Bbb HP^3, g_{_{QK}}) = 32
%\end{eqnarray*}
%by Corollary \ref{Main-Corollary}.

%For $\rho=\frac{1}{3}$ one gets an $Sp(4)$-invariant Einstein metric not isometric to $g_{_{\Bbb S}}$. This new Einstein metric has
%\begin{eqnarray*}
%\lambda_1 &\geq& \lambda_1(\Bbb HP^3 \times S^3, g_{_{QK}}+\frac{1}{9}g_{_{\Bbb S}})\\ &=& \min(32,9\cdot 3)\\ &=& 27
%\end{eqnarray*}
%by Theorem \ref{Main-theorem}.

In case of $S^{15}$, one can also view $\Bbb S^{15}$ as $Spin(9)/Spin(7)$ by using the representation of $Spin(9)$ into $SO(16)$, where the isotropy group $Spin(7)$ sits in $Spin(9)$ as the image of the usual $Spin(7)< Spin(8)$ under an ``outer" automorphism of $Spin(8)$. Thus when $G=Spin(9), H=Spin(8), K=Spin(7)$ one can have $$\pi:\Bbb S^{15}\rightarrow (S^8, \frac{1}{4}g_{_{\Bbb S}})$$ with fibers isometric to $\Bbb S^7$.
%where $H<G$ is embedded in the usual way but $K<H$ is embedded as
Shrinking the fibers constantly by $\rho\in (0,\frac{\sqrt{14}}{8}]$,  $$\lambda_1(S^8, \frac{1}{4}g_{_{\Bbb S}})=32\leq \frac{7}{\rho^2},$$
and hence this squashed $S^{15}$ has
\begin{eqnarray*}
\lambda_1 = 32
\end{eqnarray*}
by Theorem \ref{Main-Corollary}. In this case too, our computation coincides with that of Bettiol and Piccione \cite{BP} which states more generally that
$$\lambda_1=\left\{
  \begin{array}{ll}
    32 & \hbox{if $\rho\leq \frac{\sqrt{7}}{2\sqrt{6}}$} \\
    8+\frac{7}{\rho^2} & \hbox{otherwise.}
  \end{array}
\right.$$
%By the way, for $\rho=\sqrt{\frac{3}{11}}$ one gets an $Spin(9)$-invariant Einstein metric not isometric to $g_{_{\Bbb S}}$.
%By Theorem \ref{Main-theorem} this another Einstein metric has
%\begin{eqnarray*}
%\lambda_1 &\geq& \lambda_1(S^8 \times S^7, \frac{1}{4}g_{_{\Bbb S}}+\frac{3}{11}g_{_{\Bbb S}})\\ &=& \min(4\cdot 8,\frac{11}{3}\cdot 7)\\ &=& \frac{77}{3}.
%\end{eqnarray*}

Riemannian submersions of $\Bbb S^n$ with totally geodesic fibers are completely classified by Escobales \cite{Escobales-1} and Ranjan \cite{Ranjan}.
The only remaining case is the Hopf fibration $\pi: S^{2q+1}\rightarrow \Bbb CP^q$ which also arises from the above-mentioned construction. But in this case fibers are $S^1$, so our result is not applicable.

\subsection{Projective bundles}

As a general way of obtaining a $\Bbb CP^{r-1}$-bundle, one can consider
%The general way of obtaining a Riemannian submersion with totally geodesic fibers isometric to $(\Bbb CP^{r-1},g_{_{FS}})$ would be
the projectivization $\Bbb P(E_{\Bbb C})$ of a complex vector bundle $E_{\Bbb C}$ of rank $r\geq 2$ over a base Riemannian manifold $M$. Equipping $E_{\Bbb C}$ with a hermitian metric, $E_{\Bbb C}$ is the associated bundle to the unitary frame bundle $P_{U}$. Since the left action of $U(r)$ on $\Bbb C^r$ is linear, it commutes with $\Bbb C^*$-action so that it descends to a left action on $(\Bbb C^r-\{0\})/\Bbb C^*=\Bbb CP^{r-1}$. Therefore the $\Bbb CP^{r-1}$-bundle $\Bbb P(E_{\Bbb C})$ is also an associated bundle to $P_U$.

%and let $Q$ be the principal $PU(r)=U(r)/S^1$ bundle over $M$ given by the quotient $$Fr/S^1=(Fr\times PU(r))/U(r).$$ Here $S^1$ is the center and a closed subgroup of $U(r)$, so $Q$ is not only well-defined but also have the induced $PU(r)$-connection.(If $\omega$ is the connection form on $Fr$, then $R_h^*\omega=Ad(h^{-1})\omega=\omega$ for $h\in S^1$ so that $\omega$ modulo the Lie algebra of $S^1$ is the well-defined connection on $Q$.) Using the fact that $PU(r)$ is the identity component of the isometry group of $(\Bbb CP^{r-1},g_{_{FS}})$,
Since the above left action of $U(r)$ on $\Bbb CP^{r-1}$ is isometric w.r.t. the Fubini-Study metric $g_{_{FS}}$, any $U(r)$-connection on $P_U$ endows
$$\pi:\Bbb P(E_{\Bbb C})\rightarrow M$$ with a Riemannian submersion with totally geodesic fibers isometric to $(\Bbb CP^{r-1},\rho^2g_{_{FS}})$. Using that the Ricci curvature of $(\Bbb CP^{r-1},g_{_{FS}})$ is $2r$,  Remark \ref{my_sister} (of Theorem \ref{Main-Corollary}) gives $$\lambda_1(\Bbb P(E_{\Bbb C}))=\lambda_1(M)$$
when the constant $\rho$ is in $(0,\sqrt{\frac{4r(r-1)}{(2r-3)\lambda_1(M)}}]$.

The projectivization $\Bbb P(E_{\Bbb C})$ of a rank 2 complex vector bundle $E_{\Bbb C}$ over a Riemann surface $\Sigma$ is well-known as a geometrically ruled surface of which $\Bbb CP^2\# \overline{\Bbb CP}^2$ is a kind. By taking any $U(2)$-connection on $E_{\Bbb C}$ and any Riemannian metic on $\Sigma$, we have a Riemannian submersion with totally geodesic fibers isometric to $(\Bbb CP^1,\rho^2g_{_{FS}})$, and its $\lambda_1$ gets equal to $\lambda_1(\Sigma)$ when $\rho$ gets sufficiently small.

Generalizing $\Bbb CP^2\# \overline{\Bbb CP}^2$, the blow-up of $\Bbb CP^n$ along any linear subspace $\Bbb CP^{r-1}$ for $r\geq 1$ is the projective bundle of a rank $(r+1)$ vector bundle $\mathcal{O}_{\Bbb CP^{n-r}}(1)\oplus \mathcal{O}^r_{\Bbb CP^{n-r}}$ over $\Bbb CP^{n-r}$.(Consider the linear projection from $\Bbb CP^{r-1}\subset \Bbb CP^n$ to $\Bbb CP^{n-r}\subset \Bbb CP^{n}$ disjoint from the $\Bbb CP^{r-1}$ or refer to \cite{Eisenbud}.) Interestingly, the blow-ups of $\Bbb CP^n$ along certain non-linear subvarieties also turn out to be projective bundles over a projective space, and the reader may refer to \cite{Ein, Galkin, Ray}.

One can also view $\Bbb P(E_{\Bbb C})$ as $S(E_{\Bbb C})/S^1$ where the $S^1$-action is the fiberwise quotient of the unit sphere bundle of $E_{\Bbb C}$.
If $M$ is an almost hermitian manifold, then one can take the Hermitian vector bundle $E_{\Bbb C}$ to be its tangent bundle $TM$.
In particular when $E_{\Bbb C}$ is the tangent bundle of a K\"ahler manifold $M$, a canonical connection on $P_U$ would be the Levi-Civita connection of $M$ which is torsion-free.

In the same way one can construct the quaternionic projectivization $$\Bbb P(E_{\Bbb H})=S(E_{\Bbb H})/Sp(1)$$ of a quaternionic vector bundle $E_{\Bbb H}$ of rank $r\geq 2$. By endowing $E_{\Bbb H}$ with a quaternionic Hermitian metric which always exists, the structure group of $E_{\Bbb H}$ is reduced to $$Sp(r)=GL(r,\Bbb H)\cap U(2r)$$ and $E_{\Bbb H}$ is the associated bundle to a principal $Sp(r)$-bundle $P_S$.
% it can be also regarded as a $\Bbb C^{2r}$ vector bundle with quaternionic structure $j$, i.e. a fiberwise conjugate-linear map satisfying $j^2=-id$.
Since the $\Bbb H$-linear left $Sp(r)$ action on $\Bbb H^r$ commuting with the right multiplication of $\Bbb H^*=\Bbb H\backslash\{0\}$ descends to a left action on $\Bbb HP^{r-1}$, $\Bbb P(E_{\Bbb H})$ is also associated to $P_S$.

Using the quaternionic K\"ahler metric $g_{_{QK}}$ on $\Bbb HP^{r-1}$ which is isometric under the $Sp(r)$ action and choosing any $Sp(r)$-connection on $P_S$, one can endow $\Bbb P(E_{\Bbb H})$ with a Riemannian submersion with totally geodesic fibers isometric to $(\Bbb HP^{r-1},\rho^2g_{_{QK}})$ which has Ricci curvature equal to $\frac{4(r+1)}{\rho^2}$. Hence $\lambda_1(\Bbb P(E_{\Bbb H}))$ for this metric is equal to that of its base Riemannian manifold when $\rho$ gets sufficiently small.

%If $M$ is an almost hypercomplex manifold, then one can take $E_{\Bbb Q}$ to be its tangent bundle $TM$ with any quaternionic Hermitian metric which always exists, by averaging over $Sp(1)$.)
In particular when $E_{\Bbb H}$ is the tangent bundle of a hyperk\"ahler manifold $M$, a canonical connection on $P_S$ would be the Levi-Civita connection of $M$ which is torsion-free.

\begin{Remark}
In a similar way, one may also consider the real projectivization $\Bbb P(E)$ of a real vector bundle $E$ of rank $r\geq 3$, and get the same type of conclusion. In this case $E$ is just the fiberwise $\Bbb Z_2$-quotient of $S(E)$.
\end{Remark}

\section{$\lambda_1$ comparison w.r.t. Dirichlet boundary condition}

%We shall use the notation $\stackrel{\circ}{X}:=X-\partial X$ for the ``interior" of a manifold $X$ with boundary.

\subsection{Proof of Theorem \ref{selfish}}
Recall that on a compact Riemannian manifold $(\mathcal{X},\textbf{g})$ with piecewise-smooth boundary and smooth interior $\stackrel{\circ}{\mathcal{X}}:=\mathcal{X}-\partial\mathcal{X}$, $\lambda_1$ of $\Delta_{(\mathcal{X},\textbf{g})}$ w.r.t. Dirichlet boundary condition is expressed as
$$\lambda_1=\min\{R_{\textbf{g}}(f)|f\in L_{0,1}^2(\mathcal{X})\backslash\{0\}\}$$
and the first eigenfunction vanishing at $\partial\mathcal{X}$ is a minimizer of $R_{\textbf{g}}$, which is smooth on $\stackrel{\circ}{\mathcal{X}}$. (\cite{Chavel, Courant-Hilbert, Mao})

Let $\epsilon\ll 1$ be any small positive number, and take $\varphi\in C_c^\infty({X})$
satisfying
\begin{eqnarray}\label{plead}
|\lambda_1(X)-R_{h+g_{_{\Bbb E}}}(\varphi)|<\epsilon\ \ \ \textrm{and}\ \ \ \int_X\varphi^2 d\mu_{h+g_{_{\Bbb E}}}=1.
\end{eqnarray}
By Lemma \ref{Zenith}, $|\varphi|$ belongs to $L_1^2(X)$ and satisfies
%\begin{eqnarray}\label{Chabakchabak}
%\int_X|\varphi|^2 d\mu_{h+g_{_{\Bbb E}}}=1
%\end{eqnarray}
$$R_{h+g_{_{\Bbb E}}}(-|\varphi|)\leq R_{h+g_{_{\Bbb E}}}(\varphi).$$

To cook up $\psi\in L_{0,1}^2(X^\star)$ with smaller Rayleigh quotient, the basic idea is that we first perturb $-|\varphi|$ a little to get a smooth generic-fiberwise Morse function $\tilde{\varphi}$ and then take its fiberwise Euclidean rearrangement $\tilde{\varphi}_{\bar{\star}}$. For technical reasons we need to take several modifications along the way, and make sure that the increase of Rayleigh quotients is kept arbitrarily small.

\begin{Lemma}\label{hebrew}
There exists $\hat{\varphi}\in C_c^\infty({X})$ satisfying  $$\hat{\varphi}\leq 0\ \ \ \ \textrm{and}\ \ \ \ ||\hat{\varphi}-(-|\varphi|)||_{L_1^2}<\epsilon.$$
%\ \ \ \ \ \textrm{and}\ \ \ \ ||\hat{\varphi}-(-|\varphi|)||_{\infty}<\epsilon.$$
\end{Lemma}
\begin{proof}
It can be proved by the usual method using a partition of unity and a mollifier. At any point $p\in \textrm{supp}(\varphi)$ we take an open neighborhood $U\subset {X}$ giving a local coordinate chart and having compact closure in ${X}$.
Since all such $U$ together gives an open cover of $\textrm{supp}(\varphi)$, we can extract a finite subcover $\cup_{\alpha=1}^\nu U_\alpha$.  Let $\{\eta_\alpha|\alpha=1,\cdots,\nu\}$ be a smooth partition of unity subordinate to $\{U_\alpha|\alpha=1,\cdots,\nu\}$ such that each $\textrm{supp}(\eta_\alpha)$ is compact.

By taking the convolution of $|\varphi|\eta_\alpha$ with a mollifier in a local coordinate $U_\alpha$, we can get its smooth approximation $\varphi_{\alpha}$ such that $$||\varphi_{\alpha}-|\varphi|\eta_\alpha||_{L_1^2}<\frac{\epsilon}{\nu}
\ \ \ \ \ \textrm{and}\ \ \ \  \textrm{supp}(\varphi_{\alpha})\subset\subset U_\alpha.$$
%||\varphi_{\alpha}-|\varphi|\eta_\alpha||_{\infty}<\frac{\epsilon}{\nu}$$
Defining $\hat{\varphi}$ by $-\sum_{\alpha=1}^\nu\varphi_{\alpha}$, we have
\begin{eqnarray*}
||\hat{\varphi}+|\varphi|\ ||_{L_1^2} &=& ||-\sum_{\alpha=1}^\nu\varphi_{\alpha}+\sum_{\alpha=1}^\nu\eta_\alpha|\varphi|\ ||_{L_1^2}\\
&\leq& \sum_{\alpha=1}^\nu||-\varphi_\alpha+|\varphi|\eta_\alpha||_{L_1^2}\\ &<& \epsilon.
\end{eqnarray*}
%and likewise $$||\hat{\varphi}+|\varphi|\ ||_{\infty} < \epsilon.$$

Since each $|\varphi|\eta_\alpha$ and a mollifier are nonnegative functions, $\hat{\varphi}$ is nonpositive.
\end{proof}

By the above lemma,
\begin{eqnarray}\label{SujinLee1}
||\hat{\varphi}||_{L_1^2} &<& ||\ |\varphi|\ ||_{L_1^2}+\epsilon\nonumber \\ &=& \sqrt{1+R_{h+g_{_{\Bbb E}}}(|\varphi|)}+\epsilon\nonumber  \\ &\leq& \sqrt{1+R_{h+g_{_{\Bbb E}}}(\varphi)}+\epsilon\nonumber  \\ &<& \sqrt{1+\lambda_1(X)+\epsilon}+\epsilon\nonumber  \\ &<&   \sqrt{2+\lambda_1(X)}+1.
\end{eqnarray}

From now on $C',C'',C'''$ and $C_i$ for $i\in \Bbb N$ denote some positive constants independent of $\epsilon$.

\begin{Lemma}
$|R_{h+g_{_{\Bbb E}}}(\hat{\varphi})-R_{h+g_{_{\Bbb E}}}(-|\varphi|)|<C_1\epsilon$.
\end{Lemma}
\begin{proof}
By Lemma \ref{Zenith} and (\ref{SujinLee1}),
\begin{eqnarray*}
|\int_X (|d\hat{\varphi}|^2-|d|\varphi||^2)\   d\mu_{h+g_{_{\Bbb E}}}| &\leq& ||\hat{\varphi}+|\varphi|\ ||_{L_1^2}(||\hat{\varphi}||_{L_1^2}+||\ |\varphi|\ ||_{L_1^2})\\  &\leq& ||\hat{\varphi}+|\varphi|\ ||_{L_1^2}(2||\hat{\varphi}||_{L_1^2}+||\hat{\varphi}+|\varphi|\ ||_{L_1^2})\\&<& \epsilon(2(\sqrt{2+\lambda_1(X)}+1)+\epsilon)\\ &<& C'\epsilon,
\end{eqnarray*}
and
\begin{eqnarray}\label{Kimdaewhan}
|\int_X (\hat{\varphi}^2-|\varphi|^2)\   d\mu_{h+g_{_{\Bbb E}}}|  &\leq& (|| \hat{\varphi}+|\varphi|\ ||_{L_1^2})(||\hat{\varphi}||_{L_1^2}+||\ |\varphi|\ ||_{L_1^2})\nonumber\\ &<&  C'\epsilon.
\end{eqnarray}
Putting together these and $\int_X|\varphi|^2 d\mu_{h+g_{_{\Bbb E}}}=1$, the desired inequality is deduced from Lemma \ref{Sunukjian}.
%(In fact, one could use $||\hat{\varphi}-(-|\varphi|)||_{\infty}<\epsilon$ to get (\ref{Kimdaewhan}).)
\end{proof}

We can take $\epsilon'\in (0,\epsilon)$ such that any open set $W$ satisfying  $$\{z\in X|\hat{\varphi}(z)<-\epsilon'\}\subset W\subset X$$ satisfies
$$|\int_{X-W}\hat{\varphi}^2 d\mu_{h+g_{_{\Bbb E}}}|<\epsilon\ \ \ \ \textrm{and}\ \ \ \  |\int_{X-W}|d\hat{\varphi}|^2 d\mu_{h+g_{_{\Bbb E}}}|<\epsilon.$$
%For later convenience let's take $\epsilon'>0$ smaller than $\epsilon$.
For such $W$
\begin{eqnarray}\label{simon2}
|\int_{W}\hat{\varphi}^2 d\mu_{h+g_{_{\Bbb E}}}-1| &\leq & |\int_{W}\hat{\varphi}^2 d\mu_{h+g_{_{\Bbb E}}}-\int_{X}\hat{\varphi}^2 d\mu_{h+g_{_{\Bbb E}}}| \nonumber\\ & & + |\int_{X}\hat{\varphi}^2 d\mu_{h+g_{_{\Bbb E}}}-\int_{X}\varphi^2 d\mu_{h+g_{_{\Bbb E}}}| \nonumber\\ &<&
(1+C')\epsilon
\end{eqnarray}
by (\ref{Kimdaewhan}), and hence by using Lemma \ref{Sunukjian}
$$|R_{h+g_{_{\Bbb E}}}(\hat{\varphi})-R_{h+g_{_{\Bbb E}}}(\hat{\varphi}|_W)|<C_2\epsilon$$ for any such $W$, where
$R_{h+g_{_{\Bbb E}}}(\hat{\varphi}|_W)$ denotes $$\frac{\int_W|\hat{\varphi}|^2 d\mu_{h+g_{_{\Bbb E}}}}{\int_W\hat{\varphi}^2 d\mu_{h+g_{_{\Bbb E}}}}.$$ Such a notation for the Rayleigh quotient shall be used from now on.

Take a sufficiently large open $m$-disk $D\subset \Bbb R^m$ with center at the origin such that $X\cap (\{s\}\times \Bbb R^m)$ is contained in $\{s\}\times D$ for any $s\in N$, and choose any smooth closed Riemannian $m$-manifold $(M,g)$ into which $D$ is embedded isometrically. Thus $X$ is also isometrically embedded in $N\times M$, and we may consider $\hat{\varphi}$ as a smooth function on $N\times M$ by defining it to be 0 outside of $X$. We choose a smooth function $\tilde{\varphi}$ on $N\times M$ which is a generic-fiberwise Morse function and satisfies $$||\tilde{\varphi}-\hat{\varphi}||_{C^2}< \frac{\epsilon'}{10}$$ where the $C^2$-norm is computed using $h+g$.
\begin{Lemma}\label{bobby}
$$|R_{h+g_{_{\Bbb E}}}(\tilde{\varphi}|_W)-R_{h+g_{_{\Bbb E}}}(\hat{\varphi}|_W)|<C_3\epsilon$$ for any open subset $W\subset X$ containing $\{z\in X|\hat{\varphi}(z)<-\epsilon'\}$.
\end{Lemma}
\begin{proof}
It's just a trivial check following from the $C^1$-closeness of two functions on a set of finite volume $\leq \mu(X)$.
From $\sup|\tilde{\varphi}-\hat{\varphi}|< \frac{\epsilon'}{10}$ and $\sup|d\tilde{\varphi}-d\hat{\varphi}|< \frac{\epsilon'}{10}$,
it follows that
$$||\tilde{\varphi}-\hat{\varphi}||_{L_1^2(X)}< \frac{\epsilon'}{10}\sqrt{2\mu(X)}.$$
%Here and below the $L_1^2$-norms are computed in $W$.
By using Lemma \ref{Zenith} and (\ref{SujinLee1})
\begin{eqnarray*}
|\int_W (|d\tilde{\varphi}|^2-|d\hat{\varphi}|^2)\   d\mu_{h+g_{_{\Bbb E}}}| &\leq& ||\tilde{\varphi}-\hat{\varphi}||_{L_1^2(W)}(||\tilde{\varphi}||_{L_1^2(W)}+||\hat{\varphi}||_{L_1^2(W)})\\
&\leq& ||\tilde{\varphi}-\hat{\varphi}||_{L_1^2(X)}(2||\hat{\varphi}||_{L_1^2(X)}+||\tilde{\varphi}-\hat{\varphi}||_{L_1^2(X)})\\
&<& \frac{\epsilon'}{10}\sqrt{2\mu(X)}\left(2(\sqrt{2+\lambda_1(X)}+1)+\frac{\epsilon'}{10}\sqrt{2\mu(X)}\right)\\ &\leq& C''\epsilon',
\end{eqnarray*}
and
\begin{eqnarray}\label{simon1}
|\int_W (\tilde{\varphi}^2-\hat{\varphi}^2)\   d\mu_{h+g_{_{\Bbb E}}}| &\leq& (||\tilde{\varphi}-\hat{\varphi}||_{L_1^2(W)})(||\tilde{\varphi}||_{L_1^2(W)}+||\hat{\varphi}||_{L_1^2(W)})\nonumber\\
&<& C''\epsilon'.
\end{eqnarray}

%$$\int_{W}|d\hat{\varphi}|^2 d\mu_{h+g_{_{\Bbb E}}}-C'\epsilon'\leq \int_{W}|d\tilde{\varphi}|^2 d\mu_{h+g_{_{\Bbb E}}}\leq \int_{W}|d\hat{\varphi}|^2 d\mu_{h+g_{_{\Bbb E}}}+C'\epsilon'.$$
%In the same way using $\sup_X|\tilde{\varphi}-\hat{\varphi}|\leq \frac{\epsilon'}{10}$, one can get
%\begin{eqnarray}\label{simon1}
%\int_{W}\hat{\varphi}^2 d\mu_{h+g_{_{\Bbb E}}}-C''\epsilon'\leq \int_{W}\tilde{\varphi}^2 d\mu_{h+g_{_{\Bbb E}}}\leq \int_{W}\hat{\varphi}^2 d\mu_{h+g_{_{\Bbb E}}}+C''\epsilon'.
%\end{eqnarray}
Combining these with (\ref{simon2}) and Lemma \ref{Sunukjian} produces the desired inequality.
\end{proof}

Now we choose $W$ to be $\{z\in N\times M|\tilde{\varphi}(z)<-\frac{\epsilon'}{2}\}$ so that
$$\{z\in X|\hat{\varphi}(z)<-\epsilon'\}\subset W\subset X.$$
By Theorem \ref{Young&Hyuk} the fiberwise Euclidean arrangement $\tilde{\varphi}_{\bar{\star}}$ of $\tilde{\varphi}$ on $N\times M$ satisfies
$$\tilde{\varphi}_{\bar{\star}}\in L_1^2(N\times D^m_V)\cap C^0(N\times D^m_V)$$ where $V$ is the volume of $(M,g)$ and for
$W^\star:=\{z\in N\times D^m_V|\tilde{\varphi}_{\bar{\star}}(z)<-\frac{\epsilon'}{2}\}$
\begin{eqnarray*}\label{vvvv}
\int_W|d\tilde{\varphi}|^2 d\mu_{h+g_{_{\Bbb E}}}\geq \int_{W^\star}|d\tilde{\varphi}_{\bar{\star}}|^2 d\mu_{h+g_{_{\Bbb E}}},
\end{eqnarray*}
which implies that $$R_{h+g_{_{\Bbb E}}}(\tilde{\varphi}|_W)\geq R_{h+g_{_{\Bbb E}}}(\tilde{\varphi}_{\bar{\star}}|_{W^\star}).$$

Finally we define $\psi:X^\star\rightarrow \Bbb R$ by
$$\psi(y)=\left\{
  \begin{array}{ll}
    \tilde{\varphi}_{\bar{\star}}(y)+\frac{\epsilon'}{2} & \hbox{for $y\in W^\star$} \\
    0 & \hbox{for $y\in X^\star-W^\star$.}
  \end{array}
\right.$$
Since $\min(f(x),g(x))$ for continuous real-valued functions $f,g$ defined on any topological space is continuous, $\psi\in C^0(X^\star)$ and also $\psi\in L_1^2(X^\star)$ by Lemma \ref{Zenith}. Moreover
\begin{Lemma}
$\psi\in L_{0,1}^2(X^\star)$.
\end{Lemma}
\begin{proof}
%Since $\tilde{\varphi}_{\bar{\star}}$ is (uniformly) continuous on $N\times D^m_V$, (\ref{libraryinLondon-1}) holds for any $s\in N$ so that
It suffices to show that the compact set $\textrm{supp}(\psi)$ is a subset of the open set $X^\star$.
Since $\textrm{supp}(\hat{\varphi})$ is a compact subset of the open set $X$, we can choose an open subset $Y\subset X$ such that $Y$ contains $\textrm{supp}(\hat{\varphi})$ and the $m$-dimensional volume $$v(s):=\mu(Y\cap (\{s\}\times \Bbb R^m))$$ is
%continuous w.r.t $s\in N$ and
less than $V(s)-\delta$ for a constant $\delta\in (0,\min_{s\in N}V(s))$. Let's  define $$Y^\star:=\{(s,q)\in N\times \Bbb R^m|q\in D^m_{v(s)}\}.$$
For any $s\in N_0=N_0(\tilde{\varphi})$
\begin{eqnarray*}
\mu(\{z\in \{s\}\times D^m_V|\tilde{\varphi}_{\bar{\star}}(z)< -\frac{\epsilon'}{2}\}) &=&
\mu(\{z\in \{s\}\times M|\tilde{\varphi}(z)< -\frac{\epsilon'}{2}\}) \\ &<&
\mu(Y\cap (\{s\}\times M))\\ &=& \mu(Y^\star\cap (\{s\}\times D^m_V))
\end{eqnarray*}
where the inequality of the 2nd line is obtained by the fact that  $$|\tilde{\varphi}|< \frac{\epsilon'}{10} \ \ \ \ \textrm{on}\ N\times M-Y.$$
Therefore $$\{z\in N_0\times D^m_V|\tilde{\varphi}_{\bar{\star}}(z)< -\frac{\epsilon'}{2}\}\subset Y^\star,$$ and hence
\begin{eqnarray*}
\textrm{supp}(\psi)&=& \overline{\{z\in N\times D^m_V|\tilde{\varphi}_{\bar{\star}}(z)< -\frac{\epsilon'}{2}\}} \\
&=& \overline{\{z\in N_0\times D^m_V|\tilde{\varphi}_{\bar{\star}}(z)< -\frac{\epsilon'}{2}\}} \\
&\subset& \overline{Y^\star} \\
&\subset& \{(s,q)\in N\times \Bbb R^m|q\in \overline{D^m_{V(s)-\delta}}\} \\
&\subset& X^\star
\end{eqnarray*}
where the 2nd equality is due to the continuity of $\tilde{\varphi}_{\bar{\star}}$.
\end{proof}

\begin{Lemma}
$$|R_{h+g_{_{\Bbb E}}}(\psi)- R_{h+g_{_{\Bbb E}}}(\tilde{\varphi}_{\bar{\star}}|_{W^\star})|<C_4\epsilon.$$
\end{Lemma}
\begin{proof}
Since $$R_{h+g_{_{\Bbb E}}}((\tilde{\varphi}_{\bar{\star}}+\frac{\epsilon'}{2})|_{W^\star})=R_{h+g_{_{\Bbb E}}}(\psi),$$
we shall show $$|R_{h+g_{_{\Bbb E}}}((\tilde{\varphi}_{\bar{\star}}+\frac{\epsilon'}{2})|_{W^\star})- R_{h+g_{_{\Bbb E}}}(\tilde{\varphi}_{\bar{\star}}|_{W^\star})|<C_4\epsilon.$$
This is also the consequence of the fact that two functions are $C^1$-close on a set of finite volume $\leq \mu(X^\star)=\mu(X)$ as in Lemma \ref{bobby}.
By using (\ref{simon2}) and (\ref{simon1})
%, and $\epsilon'<\epsilon$
\begin{eqnarray*}
\int_{W^\star}(\tilde{\varphi}_{\bar{\star}})^2d\mu_{h+g_{_{\Bbb E}}} &=& \int_{W}\tilde{\varphi}^2d\mu_{h+g_{_{\Bbb E}}}\\ &\in& (1-(1+C'+C'')\epsilon, 1+(1+C'+C'')\epsilon).
\end{eqnarray*}
By the Minkowski inequality
\begin{eqnarray*}
|(\int_{W^\star}(\tilde{\varphi}_{\bar{\star}}+\frac{\epsilon'}{2})^2d\mu_{h+g_{_{\Bbb E}}})^{\frac{1}{2}}-(\int_{W^\star}(\tilde{\varphi}_{\bar{\star}})^2d\mu_{h+g_{_{\Bbb E}}})^{\frac{1}{2}}| &\leq& (\int_{W^\star}(\frac{\epsilon'}{2})^2d\mu_{h+g_{_{\Bbb E}}})^{\frac{1}{2}}.
\end{eqnarray*}
Using $|a^2-b^2|=|a+b| |a-b|$ and the above two estimates, we get
$$|\int_{W^\star}(\tilde{\varphi}_{\bar{\star}}+\frac{\epsilon'}{2})^2d\mu_{h+g_{_{\Bbb E}}}-\int_{W^\star}(\tilde{\varphi}_{\bar{\star}})^2d\mu_{h+g_{_{\Bbb E}}}|
< C'''\epsilon'.$$
Now the desired inequality is obtained by invoking Lemma \ref{Sunukjian}.
\end{proof}

Combining all we have obtained so far, we have
\begin{eqnarray*}
\lambda_1(X)&\geq& R_{h+g_{_{\Bbb E}}}(\varphi)-\epsilon
\\ &\geq& R_{h+g_{_{\Bbb E}}}(|\varphi|)-\epsilon
\\ &\geq& R_{h+g_{_{\Bbb E}}}(\hat{\varphi})-(1+C_1)\epsilon
\\ &\geq& R_{h+g_{_{\Bbb E}}}(\hat{\varphi}|_W)-(1+C_1+C_2)\epsilon
\\ &\geq& R_{h+g_{_{\Bbb E}}}(\tilde{\varphi}|_W)-(1+C_1+C_2+C_3)\epsilon
\\ &\geq& R_{h+g_{_{\Bbb E}}}(\tilde{\varphi}_{\bar{\star}}|_{W^\star})-(1+C_1+C_2+C_3)\epsilon
\\ &\geq&
%R_{h+g_{_{\Bbb E}}}((\tilde{\varphi}_{\bar{\star}}+\frac{\epsilon'}{2})|_{W^\star})-(1+C_1+C_2+C_3+C_4)\epsilon \\ &=&
R_{h+g_{_{\Bbb E}}}(\psi)-(1+C_1+C_2+C_3+C_4)\epsilon
\\ &\geq& \lambda_1(X^\star)-(1+C_1+C_2+C_3+C_4)\epsilon.
\end{eqnarray*}
Since $\epsilon>0$ is arbitrary, $\lambda_1(X)\geq \lambda_1(X^\star)$ must hold.

\begin{Remark}
The analogous statement for the hyperbolic symmetrization can be proved in the same way as above. Namely when $X$ is in $(N\times \Bbb H^m,h+g_{_{\Bbb H}})$ and $X^\bullet$ is defined as $$X^\bullet:=\{(s,q)\in N\times \Bbb H^m|q\in \mathfrak{D}^m_{V(s)}\}$$ where $\mathfrak{D}^m_{V(s)}$ is the open ball in $(\Bbb H^m,g_{_{\Bbb H}})$ with center at the origin and volume equal to $V(s)=\mu(X\cap(\{s\}\times \Bbb H^m))$, we have
$$\lambda_1(X)\geq \lambda_1(X^\bullet).$$
\end{Remark}

\subsection{Proof of Theorem \ref{Kane}}

The proof is almost the same as the previous subsection except we use $S^m_V$ instead of $D^m_V$. We adopt the same notation, and explain where to change in the above proof.

First we need to define an auxiliary space $$X^\dagger:=\{(s,q)\in N\times S^m_V|r(q)< \tilde{R}(s)\}$$
with metric induced from the Riemannian product $(N,\textbf{h}_V:=(\frac{V}{V_m})^{\frac{2}{m}}h)\times S^m_V$ where $r(q)$ is the distance of $q$ and the south pole and
$\tilde{R}(s)$ is defined by the condition that the volume of a geodesic ball of radius $\tilde{R}(s)$ in $S^m_V$ is $V(s)$.
It is just the conformal rescaling of $X^*$; namely, a manifold
$X^\dagger$ with another metric $\left(\frac{V}{V_m}\right)^{-\frac{2}{m}}(\textbf{h}_V+\textsl{g}_V)$ is equal to $(X^*, h+g_{_{\Bbb S}})$.

Just as $W^\star$ in the previous subsection we define $$W^*:=\{z\in N\times S^m_V|\tilde{\varphi}_{\bar{*}}(z)<-\frac{\epsilon'}{2}\}$$ contained in $X^\dagger$.
By using the inequalities in Theorem \ref{main-estimates} along with the following two obvious inequalities
\begin{eqnarray*}
\left( \frac{V_m}{V} \right)^{\frac{n}{m}}\int_{N}f\ d\mu_{\textbf{h}_V}=\int_Nf\ d\mu_h, \ \ \ \ \ \  \ \ \ \left(\frac{V}{V_m}\right)^{\frac{2}{m}}| df|_{\textbf{h}_V}^2=| df|_{h}^2
\end{eqnarray*}
for any smooth $f:N\rightarrow \Bbb R$, we get
\begin{eqnarray*}
\int_W|d\tilde{\varphi}|_{h+g}^2 d\mu_{h+g}\geq \left(\frac{V_m}{V}\right)^{\frac{n}{m}}\left(\frac{V}{V_m}\right)^{\frac{2}{m}}\int_{W^*}|d\tilde{\varphi}_{\bar{*}}|^2_{\textbf{h}_V+\textsl{g}_V} d\mu_{\textbf{h}_V+\textsl{g}_V}.
\end{eqnarray*}
Combining it with
$$\int_W\tilde{\varphi}^2 d\mu_{h+g}= \left(\frac{V_m}{V}\right)^{\frac{n}{m}}\int_{W^*}(\tilde{\varphi}_{\bar{*}})^2 d\mu_{\textbf{h}_V+\textsl{g}_V},$$
gives $$R_{h+g}(\tilde{\varphi}|_W)\geq \left(\frac{V}{V_m}\right)^{\frac{2}{m}}R_{\textbf{h}_V+\textsl{g}_V}(\tilde{\varphi}_{\bar{*}}|_{W^\star}).$$
Everything else is the same, and hence one should get
\begin{eqnarray*}
\lambda_1(X)&\geq& R_{h+g}(\varphi)-\epsilon
%\\ &\geq& R_{h+g}(|\varphi|)-\epsilon
%\\ &\geq& R_{h+g}(\hat{\varphi})-(1+C_1)\epsilon
%\\ &\leq& R_{h+g}(\hat{\varphi}|_W)-(1+C_1+C_2)\epsilon
\\ &\geq& R_{h+g}(\tilde{\varphi}|_W)-(1+C_1+C_2+C_3)\epsilon
\\ &\geq& \left(\frac{V}{V_m}\right)^{\frac{2}{m}}R_{\textbf{h}_V+\textsl{g}_V}(\tilde{\varphi}_{\bar{*}}|_{W^*})-(1+C_1+C_2+C_3)\epsilon
\\ &\geq&
%\left(\frac{V}{V_m}\right)^{\frac{2}{m}}R_{\textbf{h}_V+\textsl{g}_V}((\tilde{\varphi}_{\bar{*}}+\frac{\epsilon'}{2})|_{W^*})-(1+C_1+C_2+C_3+C_4)\epsilon \\ &=&
\left(\frac{V}{V_m}\right)^{\frac{2}{m}}R_{\textbf{h}_V+\textsl{g}_V}(\psi)-(1+C_1+C_2+C_3+C_4)\epsilon
\\ &\geq& \left(\frac{V}{V_m}\right)^{\frac{2}{m}}\lambda_1(X^\dagger)-(1+C_1+C_2+C_3+C_4)\epsilon,
\end{eqnarray*}
and hence
\begin{eqnarray*}
\lambda_1(X)&\geq& \left(\frac{V}{V_m}\right)^{\frac{2}{m}}\lambda_1(X^\dagger)\\ &=& \lambda_1(X^*).
\end{eqnarray*}
This completes the proof.

\subsection{Examples}\label{selfish-man}

We now provide some examples for Theorem \ref{selfish}.

%Consider a smooth surface $\Sigma$ of genus $g$ embedded in $\{(x,y,z)\in \Bbb R^3| -2\leq x\leq 2\}$ with two cylindrical ends given by  $$\{(x,y,z)\in \Sigma| 1\leq |x|\leq 2\}=\{(x,\cos \theta, \sin \theta)|\theta\in [0,2\pi), 1\leq |x|\leq 2\}.$$ There is an isometric action of $\Bbb Z$ on $\Bbb R^3$ generated by the translation by 4 along the $x$-axis, i.e. $(x,y,z)\mapsto (x+4,y,z)$, and its quotient space $\Bbb R^3/\Bbb Z$ with the induced metric is isomeric to $S^1_4\times \Bbb R^2$ where $S^1_4$ is the circle of length 4. Let $\tilde{\Sigma}$ be the image of $\Sigma$ under this quotient map, and $X_\Sigma$ be the bounded domain in $S^1_4\times \Bbb R^2$ enclosed by $\tilde{\Sigma}$ which is a closed surface of genus $g+1$. The volume function $V(s)$ for $s\in S^1_4$ is obviously continuous. We claim that it is also piecewise-smooth with only finitely many singularities. So $(X_\Sigma)^\star$ is diffeomorphic to the interior of a solid torus with piecewise-smooth boundary.

%To prove this, note that a local coordinate function  $s\in S^1_4$ is a Morse function.  around  any $s_0\in S^1_4$, we consider a local coordinate function $s: (s_0-\epsilon,s_0+\epsilon)\rightarrow \Bbb R$. It is well-known that this function $s$ is a Morse function.  on $\tilde{\Sigma}$

Our first example is a circular tube with varying width.
Let $\gamma_i(s)=(s,y_i(s),0)$ for $s\in [-2,2]$ and $i=1,2$ be a smooth curve in the $xy$-plane such that $y_i(s)=c_i$ near $s=-2,2$ for a constant $c_i$ and satisfy $0< y_1(s)<y_2(s)$ for all $s$. Let $\Sigma_i\subset \Bbb R^3$ for $i=1,2$ be a surface of revolution obtained by rotating $\gamma_i$ around the $x$-axis.

Consider an isometric action of $\Bbb Z$ on $\Bbb R^3$ generated by the translation by 4 along the $x$-axis, i.e. $(x,y,z)\mapsto (x+4,y,z)$, and its quotient space $\Bbb R^3/\Bbb Z$ with the induced metric is isomeric to $S^1_4\times \Bbb R^2$ where $S^1_4$ is the circle of length 4.

Let $\tilde{\Sigma}_i$ be the image of $\Sigma_i$ under this quotient map, and $X_{\Sigma_i}$ be the bounded domain in $S^1_4\times \Bbb R^2$ enclosed by $\tilde{\Sigma}_i$. Now we define $X$ to be $X_{\Sigma_2}-\overline{X_{\Sigma_1}}$ which is diffeomorphic to the product of a circle and an annulus.

In this case $V(s)=\pi (y_2(s)^2-y_1(s)^2)$ is smooth for all $s$, so $X^\star$ is diffeomorphic to the interior of a solid torus with smooth boundary. It is interesting and practical to have $\lambda_1(X)\geq \lambda_1(X^\star)$.

To obtain an example where $X$ is not a fiber bundle, we define $X$ to be $X_{\Sigma_2}$ minus a finite number of solid balls. So $X$ is like a doughnut made of Swiss cheese. More generally, the deleted parts need not be balls. They can be any compact 3-manifolds with boundary, which is embeddable in $\Bbb R^3$.  The cross-sectional area $V(s)$ at $s$ is smooth except for finite number of $s$, since it may become non-smooth at the points where the topology of a cross section changes. In this case of $X$, $X^\star$ is diffeomorphic to the interior of a solid torus with piecewise-smooth boundary.

\section{Open problems}

We finish this paper by discussing some open problems related to our results on $\lambda_i$.  First a progress in the isoperimetric problem can bring a substantial extension of our results. Let $\kappa$ be a nonpositive constant and $M_\kappa^m$ be the $m$-dimensional simply-connected space form of sectional curvature $\kappa$.

A well-known conjecture\footnote{It is also called the generalized Cartan–Hadamard conjecture, because it is about isoperimetric inequalities on Caratan-Hadamard manifolds.} of Aubin \cite{Aubin-0} states that if $(M^m,g)$ is a smooth complete simply-connected Riemannian $m$-manifold with sectional curvature $\leq \kappa$, then the ``boundary area" of a bounded domain with smooth boundary in $M$ is not less than that of a geodesic ball with the same volume in $M_{\kappa}^m$. This conjecture for $\kappa=0$ in lower dimensions $m\leq 4$ was settled by the respective work of Weil, Kleiner, and Croke \cite{Weil, Kleiner, croke2}, and the case of $\kappa=-1$ for $m\leq 3$ was resolved by Bol \cite{Bol} and Kleiner \cite{Kleiner} respectively. There is a positive expectation for higher dimensions too in view of the positive result for a domain with sufficiently small volume, obtained by F. Morgan and D. Johnson \cite{morgan-johnson}.

When the $\kappa=0$ case of Aubin's conjecture is true, one can immediately generalize Theorem \ref{Young&Hyuk} and hence Theorem \ref{selfish} too.
Namely Theorem \ref{Young&Hyuk} would be extended to the case when $\mathfrak{B}$ is isometric to a bounded domain in a smooth complete simply-connected Riemannian $m$-manifold with nonpositive sectional curvature. One can check the proof in \cite{Sung} and see that the isoperimetric condition is only needed for the inequality of Theorem \ref{Young&Hyuk} to hold in such a general manifold. Once one has the inequality
\begin{eqnarray*}
\int_{F^{-1}_{\bar{\star}}(a,b]} |dF_{\bar{\star}}|^2\ d\mu_{h+g_{_{\Bbb E}}}
\leq
\int_{F^{-1}(a,b]} |dF|^2\ d\mu_{h+g},
\end{eqnarray*}
the proof of Theorem \ref{selfish} can be transferred to this general case without any essential change.
Thus we conjecture :
\begin{Conjecture}
Let $(N^n,h)$ and $(M^m,g)$ be smooth complete Riemannian manifolds such that $N$ is closed and $M$ is simply-connected with sectional curvature $K_g\leq 0$.
Suppose that $X$ is a bounded domain with piecewise-smooth boundary in the Riemannian product $N\times M$ such that
the $m$-dimensional volume $V(s)$ of $X\cap (\{s\}\times M)$ for each $s\in N$ is a piecewise-smooth continuous positive function of $s\in N$.
Then for $$X^\star:=\{(s,q)\in N\times \Bbb R^m|q\in D^m_{V(s)}\}$$
the first Dirichlet eigenvalue satisfies $$\lambda_1(X)\geq \lambda_1(X^\star).$$
\end{Conjecture}

Likewise if the $\kappa=-1$ case of  Aubin's conjecture is true, then one can generalize the hyperbolic version of Theorem \ref{Young&Hyuk}, thereby implying the following :
\begin{Conjecture}
Let $(N^n,h)$ and $(M^m,g)$ be smooth complete Riemannian manifolds such that $N$ is closed and $M$ is simply-connected with sectional curvature $K_g\leq -1$.
Suppose that $X$ is a bounded domain with piecewise-smooth boundary in the Riemannian product $N\times M$ such that
the $m$-dimensional volume $V(s)$ of $X\cap (\{s\}\times M)$ for each $s\in N$ is a piecewise-smooth continuous positive function of $s\in N$.
Then for $$X^\bullet:=\{(s,q)\in N\times \Bbb H^m|q\in \mathfrak{D}^m_{V(s)}\}$$
%r_{_{\Bbb H}}(q)< R(s)\}$$ where $r_{_{\Bbb H}}(q)$ is the distance of $q$ from the origin in $\Bbb H^m$ and $R(s)$ is defined by the condition that the volume of a geodesic ball of radius $R(s)$ in $\Bbb H^m$ is $V(s)$,
the first Dirichlet eigenvalue satisfies $$\lambda_1(X)\geq\lambda_1(X^\bullet).$$
\end{Conjecture}
So both of our conjectures are already resolved in lower dimensions by combining our fiberwise symmetrization methods and the resolution of Aubin's conjectures in the corresponding dimensions. For a convincing evidence in higher dimensions, one may consider Cheng's another comparison result \cite{Cheng2} stating that for a smooth complete Riemannian $m$-manifold $(M^m,g)$ with sectional curvature $K_g\leq \kappa$ $$\lambda_1(B_r^M)\geq\lambda_1(B_r^{M^m_\kappa})$$ for any geodesic ball $B_r^M$ of radius $r$ embedded in $M$.
%where $B_r^M$ and $B_r^{M_K}$ denote a geodesic ball of radius $r$ in $(M,g)$ and $M_K$ respectively.
For another evidence, Mckean's theorem \cite{mckean} asserts that for a smooth simply-connected complete Riemannian $m$-manifold $(M^m,g)$ with sectional curvature $K_g\leq -1$ $$\lambda_1(M)\geq \lambda_1(\Bbb H^m)=\frac{(m-1)^2}{4}$$ where $\lambda_1$ of a noncompact Riemannian manifold is defined as $\lim_{r\rightarrow \infty} \lambda_1(B_r)$, i.e. the limit of Dirichlet eigenvalues of any large geodesic balls.
%Based on the evidence in lower dimensions, we hope for their resolution.
%In this case we already have its proof when $m$ is 2 or 3. The proof of Theorem \ref{selfish} can be applied by using the corresponding isoperimetric comparison theorems of Bol \cite{Bol} and Kleiner \cite{Kleiner} in dimension 2 and 3 respectively.

Secondly, a natural follow-up study after this paper is to extend comparison results of $\lambda_i$ to a fiber bundle which is not associated to a principal $G$-bundle for a compact Lie group $G$ and to investigate whether higher Dirichlet eigenvalues can be also estimated by using our symmetrization method.
%We hope to clarify this in our next project

Lastly, on $\Bbb R^m$ there are many results on Faber-Krahn type inequalities of $\lambda_1$ for less symmetric domains of $\Bbb R^m$ such as bounded domains with nontrivial topology.(Refer to \cite{Henrot} and references therein.) In some cases, the most symmetric ones among domains with the same topology and volume turn out to have the maximum or minimum 1st Dirichlet eigenvalues. So it's a natural attempt to extend these results to a corresponding fiber bundle by using the fiberwise symmetrization developed in the present paper or \cite{morgan-H-H} and references therein.

\bigskip

\noindent{\bf Acknowledgement} The author would like to express his deepest gratitude to God, parents, and all the teachers who taught him mathematics.
The bulk of this work was done while on sabbatical, so the author is grateful to his home country for allowing him such a valuable opportunity.

\vspace{0.5cm}

\end{document}